\documentclass[11pt,reqno]{amsart}

\usepackage{verbatim}
\usepackage{xcolor}
\usepackage{graphicx}
\usepackage{amsmath}
\usepackage[margin=2.75cm]{geometry}
\allowdisplaybreaks

\usepackage{amssymb}
\usepackage[colorlinks=true, urlcolor=blue, linkcolor=blue, citecolor=blue, pdfstartview=FitH]{hyperref}
\usepackage{mathrsfs}
\usepackage{mathtools}
\usepackage[utf8]{inputenc}
\usepackage{parskip}

\newcommand{\Gauss}{\mathcal K}

\newcommand{\ben}{\begin{eqnarray*}}
\newcommand{\een}{\end{eqnarray*}}
\newcommand{\inner}[1]{\left\langle{#1}\right\rangle}

\newcommand{\R}{\mathbb{R}}

\makeatletter
\@namedef{subjclassname@2020}{\textup{2020} Mathematics Subject Classification}
\makeatother

\DeclareMathOperator{\arccot}{arccot}

\mathtoolsset{showonlyrefs}
\newtheorem{theorem}{Theorem}
\newtheorem{corollary}[theorem]{Corollary}
\newtheorem{remark}[theorem]{Remark}

\newtheorem{proposition}[theorem]{Proposition}
\newtheorem{lemma}[theorem]{Lemma}
\newtheorem{definition}{Definition}

\begin{document}

\date{\today}

\title{Heintze-Karcher and Reverse Alexandrov-Fenchel Inequalities via Focal Geometry}
\author[K-K. Kwong]{Kwok-Kun Kwong}
\address{School of Mathematics and Physics\\
University of Wollongong\\
Northfields Ave, Wollongong, NSW $2500$\\
Australia}
\email{kwongk@uow.edu.au}

\author[S. Parkins]{Scott Parkins}
\address{Independent Researcher, New South Wales, Australia}
\email{scott.parkins@aecom.com}

\author[G. Wheeler]{Glen Wheeler}
\address{School of Mathematics and Physics\\
University of Wollongong\\
Northfields Ave, Wollongong, NSW $2500$\\
Australia}
\email{glenw@uow.edu.au}


\begin{abstract}
We prove a collection of reverse Alexandrov-Fenchel type inequalities in anisotropic, Euclidean, spherical, and hyperbolic settings. The unifying principle is that the relevant deficit is controlled by curvature radius data, or equivalently by the signed volume of an associated evolute or focal map.

For smooth simple strictly convex curves in a smooth Minkowski plane we prove an anisotropic Hurwitz-type inequality: the anisotropic isoperimetric deficit is bounded above by the signed Euclidean area of the Minkowski evolute. For smooth closed strictly convex hypersurfaces in \(\mathbb R^{n+1}\), with $E_k$ denoting the normalised $k$-th mean curvature, we establish the sharp reverse Alexandrov-Fenchel estimate
\[
0\le
\frac{1}{|\mathbb S^{n}|} \left(\int_{M}E_{n-1}\,d\mu\right)^{2} -\int_{M}E_{n-2}\,d\mu
\le
\frac{n}{2(n+1)} \int_{M}\frac{E_{n-1}^{2}-E_{n-2}E_{n}}{E_{n}}\,d\mu .
\]
We also relate the deficit of the Minkowski inequality to the oriented volumes of the two focal maps.

In space forms we derive a normal-graph formula for oriented volume and use it to give focal-map interpretations of the deficit of an unweighted Heintze--Karcher inequality. In dimension two this recovers evolute-area formulae. We then prove exact reverse isoperimetric identities for curves on \(\mathbb S^{2}\) and strictly horoconvex curves in \(\mathbb H^{2}\), in which the remainders are explicit nonnegative integrals measuring the oscillation of geodesic curvature. In the spherical case, for every smooth simple closed curve $\gamma\subset\mathbb S^2$ with length $L$ and enclosed area \(A\), if $k_E$ denotes the ambient Euclidean curvature of $\gamma$, then
\[
L^{2}-A(4\pi-A)
\le
\left(\int_{\gamma}k_E\,ds\right)^{2}-4\pi^{2},
\]
with equality if and only if \(\gamma\) is a geodesic circle. This relates the spherical isoperimetric deficit with the Euclidean Fenchel deficit.
\end{abstract}
\maketitle

\tableofcontents

\section{Introduction}

The isoperimetric inequality is both a comparison theorem and a rigidity theorem. It says that among regions of prescribed volume the model ball has the least boundary measure, and that equality characterises the model. A natural quantitative problem is to understand the corresponding deficit. In general one seeks lower bounds for the deficit in terms of distance to the class of extremisers. In convex geometry there is also a complementary question: can the deficit be bounded from above by a geometric quantity naturally attached to the boundary?

The guiding example for this paper is Hurwitz's reverse isoperimetric inequality for convex plane curves. If \(f:\mathbb S^{1}\to\mathbb R^{2}\) is a convex curve and \(e\) is its evolute, Hurwitz proved that
\begin{equation}
L[f]^2-4\pi A[f]\le \pi |A[e]|,\label{HurwitzInequality}
\end{equation}
where \(A[e]\) denotes the algebraic area enclosed by the evolute. Thus the Euclidean isoperimetric deficit is controlled by a signed area associated with the focal geometry of the curve. This point of view has been refined in several directions; see, for example, \cite{Hurwitz1902,groemer1996geometric, CufiGallegoReventos2018,CufiReventos2014}.

The aim of this paper is to develop this focal-set viewpoint in several settings. We prove reverse inequalities for curves in Minkowski planes, for convex hypersurfaces in Euclidean space, and for curves and hypersurfaces in space forms. Although the statements differ from one setting to another, the same mechanism appears repeatedly: the deficit is controlled by curvature-radius data, by a trace-free curvature quantity, or by the oriented volume of an evolute or focal map.

We first treat curves in a smooth Minkowski plane. Let \(\mathcal U\subset \mathbb R^{2}\) be a smooth centrally symmetric strictly convex body, let \(\mathcal M^{2}=(\mathbb R^{2},l)\) be the corresponding Minkowski plane, and let \(\mathcal I\) be the isoperimetrix. For a smooth simple strictly convex curve \(\gamma\), write \(\mathscr L(\gamma)\) for its Minkowski length and \(\mathscr A(\gamma)\) for its Euclidean area. The anisotropic isoperimetric inequality reads
\[
\mathscr L(\gamma)^{2}
\ge
4\,\mathscr A(\gamma)\,\mathscr A(\mathcal I),
\]
with equality precisely on homothetic copies of the isoperimetrix. Our first main result is an anisotropic reverse inequality of Hurwitz type.

\begin{theorem}[Anisotropic reverse isoperimetric inequality]
\label{ThmReverseIsoMinkPlane}
Suppose \(\gamma:\mathbb S^{1}\to\mathcal M^{2}\) is a smooth, simple, strictly
convex curve in the Minkowski plane \(\mathcal M^{2}\), with indicatrix
\(\partial\mathcal U\) and isoperimetrix \(\mathcal I\). Let
\[
e:=\gamma+\kappa^{-1}N
\]
be the Minkowski evolute, regarded as a closed front, and let
\(\mathscr A(e)\) denote its signed Euclidean area. Then
\[
\mathscr L(\gamma)^{2}
-
4\,\mathscr A(\gamma)\,\mathscr A(\mathcal I)
\le
4\,\mathscr A(\mathcal I)\,|\mathscr A(e)|.
\]
\end{theorem}

This gives a direct anisotropic analogue of the principle behind Hurwitz's inequality: the failure of the curve to be isoperimetrically optimal is measured from above by the signed area swept out by its focal geometry. The constant \(4\mathscr A(\mathcal I)\) is not expected to be sharp in general. In the Euclidean case \(\mathscr A(\mathcal I)=\pi\), while Hurwitz's theorem gives the smaller constant \(\pi\). Determining the sharp anisotropic constant remains an interesting open problem.

We next turn to convex hypersurfaces in Euclidean space. The classical Minkowski inequality for smooth closed strictly convex surfaces \(M^{2}\subset\mathbb R^{3}\) is
\begin{equation}\label{ineq mink intro}
\int_{M}H\,d\mu\ge \sqrt{16\pi |M|},
\end{equation}
with equality only for round spheres. More generally, the Alexandrov--Fenchel inequalities give comparison inequalities for the integrals of higher order mean curvatures. In this paper we prove a reverse estimate for an Alexandrov-Fenchel deficit. Let \(E_k\) denote the normalised \(k\)-th mean curvature of a smooth closed strictly convex hypersurface \(M^{n}\subset\mathbb R^{n+1}\).

\begin{theorem}[Euclidean reverse Alexandrov-Fenchel inequality]
\label{thm higher dim reverse mink intro}
Let \(M\) be a smooth closed strictly convex hypersurface in
\(\mathbb R^{n+1}\). Then
\[
0\le
\frac{1}{|\mathbb S^{n}|} \left(\int_{M}E_{n-1}\,d\mu\right)^{2} - \int_{M}E_{n-2}\,d\mu
\le
\frac{n}{2(n+1)} \int_{M} \frac{E_{n-1}^{2}-E_{n-2}E_n}{E_n}\,d\mu .
\]
\end{theorem}
The inequality on the left is the Alexandrov-Fenchel inequality in $\mathbb R^{n+1}$ and the equality holds if and only if $M$ is a hypersphere. The Minkowki inequality \eqref{ineq mink intro} is the special case when $n=2$.
For \(n\ge2\), the right-hand side is nonnegative by the
Newton--Maclaurin inequalities. In the case \(n=1\), with the convention
\(E_{-1}=\langle f,\nu\rangle\), the right-hand side is the
Heintze--Karcher deficit.
The proof is based on the support-function representation of a strictly convex hypersurface and a sharp spectral Poincar\'e inequality on the sphere. The estimates in Theorem \ref{thm higher dim reverse mink intro} are sharp at the level of this spectral inequality, while the second inequality admits non-round equality cases. 

When \(n=2\), Theorem~\ref{thm higher dim reverse mink intro} gives the following reverse Minkowski inequality for strictly convex surfaces in \(\mathbb R^{3}\).

\begin{theorem}[Reverse Minkowski inequality in \(\mathbb R^{3}\)]
\label{ReverseMinkowskiInequalityLemma1}
Let \(f:M^{2}\to\mathbb R^{3}\) be a smooth, closed, strictly convex embedding.
Then
\[
\left(\int_{M}H\,d\mu\right)^{2}
-
16\pi |M|
\le
\frac{8\pi}{3}
\int_{M}\frac{|\stackrel \circ A|^{2}}{\Gauss}\,d\mu ,
\]
where \(\stackrel \circ A\) and \(\Gauss\) are the trace-free second fundamental form and
Gauss curvature, respectively.
\end{theorem}

This formulation makes the geometric meaning transparent: the Minkowski deficit is controlled by a scale-invariant integral measure of non-umbilicity. We also relate the same deficit to the focal maps of the surface. If \(\rho_i=\lambda_i^{-1}\) are the ordered principal radii and
\[
b_i=f-\rho_i\nu
\]
are the corresponding focal maps, then their oriented volumes satisfy estimates which control the same deficit. In dimension two this gives the explicit consequence
\[
\left(\int_M H\,d\mu\right)^2-16\pi |M|
\le
\frac{8\pi}{3}\,(4\pi)^{1/3}
\left(
\frac{3}{\sqrt2}|V[b_1]-V[b_2]|
\right)^{2/3}.
\]
Thus the Euclidean reverse Minkowski deficit can be read either as an umbilicity-controlled quantity or as a quantity controlled by the separation of the two focal volumes.

To support the focal-volume interpretation we prove a general oriented-volume formula for normal graphs in simply connected space forms. Let \(S_K^{n+1}\) denote the simply connected space form of sectional curvature \(K\in\{0,\pm1\}\), and let
\[
\Phi(x,t)=\exp_{f(x)}(t\nu(x))
\]
be the normal exponential map of a closed oriented hypersurface. For a normal graph \(x\mapsto \Phi(x,u(x))\), its oriented volume is given by an explicit Steiner-type formula \eqref{EnclosedVolumeEqn1} involving the higher order mean curvatures of \(M\). This formula is used throughout the paper for focal maps and evolutes.

We also apply this perspective to the unweighted Heintze--Karcher inequality. The Heintze--Karcher deficit is classically controlled by the geometry of the normal exponential map. Our contribution here is to derive a Heintze--Karcher identity in which the deficit is expressed directly in terms of the oriented volume of the first focal map. In particular, for a smooth closed mean-convex surface \(M=\partial\Omega\subset\mathbb R^{3}\), if \(b_1\) is the first focal map and \(\lambda_1\) is the largest principal curvature, then
\[
\frac13\int_M\frac1{E_1}\,d\mu-|\Omega|
=
|V[b_1]|
+
\frac16\int_M\frac{|\stackrel \circ A|^{2}}{\lambda_1^{2}E_1}\,d\mu .
\]
This identity separates the Heintze--Karcher deficit into a focal-volume term and an explicit non-umbilicity remainder, thereby addressing the question, raised in \cite{EscuderoReventos2007}, of how to interpret the three-dimensional Heintze--Karcher deficit geometrically.

In the non-Euclidean space forms, we derive an unweighted Heintze--Karcher inequality (Theorem \ref{thm hk deficit}) involving only the enclosed volume, the mean curvature, and the oriented volume of the first focal map. This distinguishes our formulation from the weighted versions in the literature \cite{Brendle2013, QiuXia2015}, where the inequality involves an additional conformal factor. We also apply it to obtain an Alexandrov-type rigidity result (Theorem \ref{thm Alexandrov}).

The final part of the paper concerns reverse isoperimetric inequalities in space forms involving Fenchel-type deficits. Recall that if \(\gamma\subset \mathbb S^{2}\) is a smooth simple closed curve of length \(L\), enclosing area \(A\in(0,2\pi)\), then the spherical isoperimetric inequality is
\[
L^{2}\ge A(4\pi-A).
\]
When \(\gamma\) is regarded as a space curve in \(\mathbb R^{3}\), its Euclidean curvature is \(\sqrt{1+k_g^{2}}\), where \(k_g\) is the geodesic curvature on \(\mathbb S^{2}\). We prove that the spherical isoperimetric deficit is controlled by the corresponding Euclidean Fenchel deficit.

\begin{theorem}[Reverse isoperimetric inequality on \(\mathbb S^{2}\)]
\label{SphereReverseIsoperimetricInequalityTheorem1}
Let \(\gamma:\mathbb S^{1}\to\mathbb S^{2}\) be a smooth, simple, closed curve.
Let $L$ be its length, $A$ be the area of the domain enclosed by $\gamma$ whose induced boundary orientation agrees with that of $\gamma$, and let $k_E$ denote the curvature of $\gamma$ when regarded as a space curve in $\mathbb R^3$.
Then
\[
L^{2}-A(4\pi-A)
\le
\left(\int_{\gamma}k_E\,ds\right)^{2}
-
4\pi^{2}.
\]
In fact,
\[
L^{2}-A(4\pi-A)
=
\left(\int_{\gamma}k_E\,ds\right)^{2}
-
4\pi^{2}
-
\mathcal R(\gamma),
\]
where
\[
\mathcal R(\gamma):=
\iint_{\gamma\times\gamma}
\frac{(k_g(s)-k_g(\sigma))^{2}}
{\sqrt{(1+k_g(s)^{2})(1+k_g(\sigma)^{2})}+1+k_g(s)k_g(\sigma)}
\,ds\,d\sigma
\ge0
\]
and $k_g$ is the geodesic curvature of $\gamma$ in $\mathbb S^2$.

Equality holds if and only if \(\gamma\) is a geodesic circle.
\end{theorem}

This result strengthens the known evolute-based inequalities in constant curvature by removing the convexity assumption in the spherical curve case and by giving an exact nonnegative remainder. We prove an analogous result for strictly horoconvex curves in the hyperbolic plane. If \(\gamma\subset\mathbb H^{2}\) has \(k_g>1\), then, when viewed as a strong spacelike curve in \(\mathbb R^{2,1}\), its Lorentzian curvature is
\[
k_E=\sqrt{k_g^{2}-1}.
\]
We show that
\[
L^{2}-A(4\pi+A)
\le
4\pi^{2}
-
\left(\int_{\gamma}k_E\,ds\right)^{2},
\]
again with an exact nonnegative double-integral remainder and equality precisely on geodesic circles.

For hypersurfaces in the northern hemisphere we obtain a higher-dimensional spherical analogue under horoconvexity assumptions. Combining a quermassintegral inequality with the algebraic remainder identity used in the curve case, we control the spherical isoperimetric deficit by the deficit in the Chen--Fenchel inequality. We also derive a spherical focal-map estimate: for a smooth closed embedded hypersurface \(M^n\subset\mathbb S^{n+1}\), the difference of the oriented volumes of the focal maps corresponding to the largest and smallest principal curvatures is bounded by an integral power of \(E_1^2-E_2\). This gives a spherical counterpart to the Euclidean focal-volume estimate, although the corresponding hyperbolic estimate appears to require a different idea.

The paper is organised as follows. In Section~\ref{SectionMinkowskiPlane} we review the necessary differential geometry of the Minkowski plane and prove the anisotropic reverse isoperimetric inequality. In the following section we prove the oriented-volume formula for normal graphs in space forms. Section~\ref{SectionEuclideanSpace} establishes the Euclidean reverse Alexandrov-Fenchel  inequalities and the focal-volume estimates for convex hypersurfaces. Section~\ref{sec Heintze Karcher} treats the Heintze--Karcher deficit from the focal-map viewpoint. In Section~\ref{SectionCurvesInSphere} we prove the spherical and hyperbolic reverse isoperimetric inequalities involving Fenchel deficits, derive the evolute-area reformulations, and prove the spherical focal-map estimate. The appendix collects the auxiliary Minkowski-plane facts and the spectral Poincar\'e inequality on \(\mathbb S^{n}\).

\section{Convex bodies and differential geometry of the Minkowski plane}\label{SectionMinkowskiPlane}
We introduce the fundamental concepts of convex body geometry before moving onto the Minkowski plane, the setting for our paper. To get a broader understanding of some of the finer details of the Minkowski plane and anisotropic vector spaces, the authors recommend reading the comprehensive survey articles of Martini and Swanepoel \cite{martini2004geometry,martini2001geometry}.

We fix a smooth, centrally symmetric, strictly convex body \(\mathcal U\subset\R^{2}\)
with \(0\in\operatorname{int}\mathcal U\). Its boundary can be written in polar form as
\[
\partial\mathcal U=\{r(\theta)(\cos\theta,\sin\theta):\theta\in[0,2\pi)\},
\]
where \(r(\theta)>0\) and \(r(\theta+\pi)=r(\theta)\).

The associated Minkowski functional is
\[
l(x):=\inf\{s>0:x\in s\mathcal U\},
\]
which is a norm because \(\mathcal U\) is a centrally symmetric convex body. If
\(x\neq0\) has Euclidean polar angle \(\theta(x)\), then
\[
l(x)=\frac{|x|}{r(\theta(x))}.
\]
We write \(\mathcal M^{2}=(\R^{2},l)\) for the corresponding Minkowski plane.\footnote{One should not confuse this with \(1+1\)-dimensional Minkowski spacetime.}

The polar dual body is
\[
\mathcal U^{*}:=\{y\in\R^{2}:\langle x,y\rangle\le1\ \text{for all }x\in\mathcal U\}.
\]
If \(h=h_{\mathcal U^{*}}\) denotes the support function of \(\mathcal U^{*}\), then
\[
h(\theta)=r(\theta)^{-1};
\]
see \cite{Gage1}. Since \(\partial\mathcal U\) is smooth and strictly convex, one has
\(h_{\theta\theta}+h>0\).

Let \(\gamma:\mathbb S^{1}\to\mathcal M^{2}\) be a \(C^{1}\) closed curve. Its Minkowski
length is
\[
\mathscr L(\gamma)=\int_{\mathbb S^{1}}l(\gamma_{u})\,du=\int_{\gamma}d\sigma,
\]
where, if \(s\) denotes Euclidean arclength and \(\tau=\gamma_{s}\) is the Euclidean unit tangent with angle \(\theta\), then
\[
d\sigma=\frac{ds}{r(\theta)}=h(\theta)\,ds.
\]

Throughout this section, smooth strictly convex closed curves are oriented
positively, so that their Euclidean curvature \(k\) and Minkowski curvature
\(\kappa\) are positive.

Write
\[
\tau=(\cos\theta,\sin\theta),\qquad n=(-\sin\theta,\cos\theta).
\]
We define the Minkowski tangent and Minkowski normal by
\[
T=h^{-1}\tau,\qquad N=-h_{\theta}\tau+h\,n.
\]
A direct calculation gives
\[
T\wedge N=\tau\wedge n=1,
\]
where \(x\wedge y:=x_{1}y_{2}-x_{2}y_{1}\). Accordingly, for any closed \(C^{1}\) curve
\(c\) we define its signed Euclidean area by
\[
\mathscr A(c):=\frac12\int_{\mathbb S^{1}}c\wedge c_{u}\,du.
\]
If \(c\) is a positively oriented immersed simple closed curve, then
\[
\mathscr A(c)=-\frac12\int_{c}\langle c,n\rangle\,ds,
\]
so \(\mathscr A(c)\) coincides with the ordinary enclosed Euclidean area.

The isoperimetrix \(\mathcal I\) is the image of the Minkowski normal along the
indicatrix:
\begin{equation}
\mathcal I:=\{N(\theta):\theta\in[0,2\pi)\}
=\{-h_{\theta}\tau+h\,n:\theta\in[0,2\pi)\}.
\label{EqnIsoperimetrix1}
\end{equation}

\begin{proposition}
For fixed Euclidean area, the minimiser of Minkowski length among convex sets in
\(\mathcal M^{2}\) is, up to translation and homothety, the isoperimetrix
\(\mathcal I\).
\end{proposition}
\begin{proof}
The result is standard. See, for example, \cite{chakerian1960isoperimetric}.
\end{proof}

Using \eqref{EqnIsoperimetrix1}, the chain rule, and \(\theta_{s}=k\), we obtain
\[
T_{\sigma}=k h^{-3}N,\qquad N_{\sigma}=-k(h_{\theta\theta}+h)T.
\]
We therefore define the Minkowski curvature of \(\gamma\) by
\begin{equation}
\kappa:=k(h_{\theta\theta}+h).
\label{MinkowskiCurvatureDefn}
\end{equation}
The Euclidean curvature of \(\mathcal I\) at the point \(N(\theta)\) is
\[
(h_{\theta\theta}+h)^{-1}.
\]

\subsection{Reverse isoperimetric inequality in the Minkowski plane}
The anisotropic isoperimetric ratio is the Minkowski analogue of the Euclidean
ratio \(L^{2}/(4\pi A)\). For smooth closed strictly convex curves in
\(\mathcal M^{2}\), the anisotropic isoperimetric inequality states that
\[
\mathscr L(\gamma)^{2}-2\mathscr A(\gamma)\int_{\gamma}\kappa\,d\sigma\ge0,
\]
with equality if and only if \(\gamma\) is homothetic to the isoperimetrix
\(\mathcal I\); see \cite{chakerian1960isoperimetric}. Using
Proposition~\ref{SubsectionMinkowskiAppendixPropostion1}, this is equivalent to
\begin{equation}
\mathscr L(\gamma)^{2}-4\mathscr A(\gamma)\mathscr A(\mathcal I)\ge0.
\label{MinkIsoInequalityEqn}
\end{equation}
This motivates the anisotropic isoperimetric ratio
\[
\mathscr I(\gamma):=\frac{\mathscr L(\gamma)^{2}}{4\mathscr A(\gamma)\mathscr A(\mathcal I)}\ge1,
\]
with equality if and only if \(\gamma\) is homothetic to \(\mathcal I\).

\begin{lemma}[Signed area of a Minkowski normal graph]\label{MinkGraphAreaLemma1}
Suppose that \(\gamma:\mathbb S^{1}\to\mathcal M^{2}\) is a smooth, simple,
closed, strictly convex curve in the Minkowski plane with Minkowski normal \(N\)
and Minkowski curvature \(\kappa\). For \(\phi\in C^{\infty}(\mathbb S^{1})\),
set
\[
\tilde{\gamma}:=\gamma+\phi N.
\]
Then the signed Euclidean area of \(\tilde{\gamma}\) is
\begin{equation}
\mathscr A(\tilde{\gamma})
=
\mathscr A(\gamma)
+\frac12\int_{\gamma}\kappa(\phi-\kappa^{-1})^{2}\,d\sigma
-\frac12\int_{\gamma}\kappa^{-1}\,d\sigma.
\label{MinkGraphAreaLemma1Eqn1}
\end{equation}
\end{lemma}
\begin{proof}
Since \(\tilde\gamma\) may fail to be immersed, we compute its signed area from
\[
\mathscr A(\tilde\gamma)=\frac12\int_{\gamma}\tilde\gamma\wedge\tilde\gamma_{\sigma}\,d\sigma.
\]
Using \(\gamma_{\sigma}=T\), \(N_{\sigma}=-\kappa T\), and \(T\wedge N=1\), we have
\[
\tilde{\gamma}_{\sigma}=(1-\kappa\phi)T+\phi_{\sigma}N.
\]
Therefore
\begin{align*}
2\mathscr A(\tilde{\gamma})
&=\int_{\gamma}(\gamma+\phi N)\wedge\big((1-\kappa\phi)T+\phi_{\sigma}N\big)\,d\sigma\\
&=\int_{\gamma}\gamma\wedge T\,d\sigma
-\int_{\gamma}\kappa\phi\,\gamma\wedge T\,d\sigma
+\int_{\gamma}\phi_{\sigma}\,\gamma\wedge N\,d\sigma
-\int_{\gamma}\phi(1-\kappa\phi)\,d\sigma.
\end{align*}
Next,
\[
(\gamma\wedge N)_{\sigma}
=
\gamma_{\sigma}\wedge N+\gamma\wedge N_{\sigma}
=
T\wedge N-\kappa\,\gamma\wedge T
=
1-\kappa\,\gamma\wedge T.
\]
Integrating by parts on the closed curve,
\[
\int_{\gamma}\phi_{\sigma}\,\gamma\wedge N\,d\sigma
=
-\int_{\gamma}\phi(1-\kappa\,\gamma\wedge T)\,d\sigma.
\]
Substituting back and using
\[
\mathscr A(\gamma)=\frac12\int_{\gamma}\gamma\wedge T\,d\sigma
\]
gives
\[
\mathscr A(\tilde{\gamma})
=
\mathscr A(\gamma)-\int_{\gamma}\phi\,d\sigma
+\frac12\int_{\gamma}\kappa\phi^{2}\,d\sigma.
\]
Completing the square yields \eqref{MinkGraphAreaLemma1Eqn1}.
\end{proof}

\begin{corollary}\label{MinkGraphAreaCor1}
Under the hypotheses of Lemma~\ref{MinkGraphAreaLemma1}, every Minkowski normal
graph \(\tilde{\gamma}=\gamma+\phi N\) satisfies
\begin{equation}
\mathscr A(\tilde{\gamma})\ge \mathscr A(\gamma)-\frac12\int_{\gamma}\kappa^{-1}\,d\sigma.
\label{MinkGraphAreaCor1Eqn1}
\end{equation}
Equality holds if and only if \(\phi=\kappa^{-1}\), equivalently
\(\tilde{\gamma}\) is the Minkowski evolute
\[
e:=\gamma+\kappa^{-1}N.
\]
\end{corollary}
\begin{proof}
This is immediate from Lemma~\ref{MinkGraphAreaLemma1}, since \(\kappa>0\) and
\(\int_{\gamma}\kappa(\phi-\kappa^{-1})^{2}\,d\sigma\ge0\).
\end{proof}

Theorem \ref{ThmReverseIsoMinkPlane} now follows from
Lemma~\ref{MinkGraphAreaLemma1} and the anisotropic isoperimetric inequality.

\begin{proof}[Proof of Theorem~\ref{ThmReverseIsoMinkPlane}]
Since \(\kappa>0\), the Cauchy-Schwarz inequality gives
\[
\mathscr L(\gamma)^{2}
=
\left(\int_{\gamma}1\,d\sigma\right)^{2}
\le
\left(\int_{\gamma}\kappa\,d\sigma\right)
\left(\int_{\gamma}\kappa^{-1}\,d\sigma\right).
\]
Applying Lemma~\ref{MinkGraphAreaLemma1} with \(\phi=\kappa^{-1}\) yields
\[
\frac12\int_{\gamma}\kappa^{-1}\,d\sigma
=
\mathscr A(\gamma)-\mathscr A(e).
\]
Therefore
\begin{align}
\mathscr L(\gamma)^{2}
&\le
2\left(\int_{\gamma}\kappa\,d\sigma\right)\bigl(\mathscr A(\gamma)-\mathscr A(e)\bigr)\\
&=
4\mathscr A(\mathcal I)\bigl(\mathscr A(\gamma)-\mathscr A(e)\bigr),
\label{ThmReverseIsoMinkPlaneProofEqn1}
\end{align}
where we used Proposition~\ref{SubsectionMinkowskiAppendixPropostion1} in the
second line. Combining this with the anisotropic isoperimetric inequality
\eqref{MinkIsoInequalityEqn} gives \(\mathscr A(e)\le0\). Hence
\[
\mathscr L(\gamma)^{2}-4\mathscr A(\gamma)\mathscr A(\mathcal I)
\le
-4\mathscr A(\mathcal I)\mathscr A(e)
=
4\mathscr A(\mathcal I)\,|\mathscr A(e)|,
\]
which proves the theorem.
\end{proof}
\section{Higher order mean curvatures and oriented volume in space forms}
\subsection{Higher order mean curvatures}

Although the main results of this paper concern hypersurfaces in space forms or the Minkowski plane, it is convenient to begin in the slightly broader setting of smooth closed hypersurfaces in an $(n+1)$-dimensional Riemannian manifold $N$.

Let $f:M^n\to N^{n+1}$ be a smooth, closed, oriented hypersurface with induced metric $g$, chosen unit normal $\nu$, and second fundamental form $A$.
Throughout this paper, we use the convention
$$A(X, Y)=\left\langle \overline{\nabla}_X\nu, Y\right\rangle, $$
where $\overline{\nabla}$ denotes the Levi-Civita connection of the ambient manifold $N$.

For a curve, we define its geodesic curvature by
$$k_g=A(T, T), $$
where $T$ is the unit tangent vector. Hence, if $\gamma=\partial\Omega$ is the positively oriented, connected boundary of a domain in a two-dimensional space form of curvature $K$, and $\nu$ denotes the outward unit normal along $\gamma$, then our convention yields
$$\int_\gamma k_g\, ds+K|\Omega|=2\pi. $$
In the special case of a plane curve in $\mathbb R^2$, we shall write $\kappa$ instead of $k_g$ for its curvature.

Let $\lambda_1, \dots, \lambda_n$ denote the principal curvatures of $M$. For $0\le k\le n$, define
$$
H_k:=\sigma_k(\lambda_1, \dots, \lambda_n),
\qquad
H_0:=1,
$$
where $\sigma_k$ is the $k$-th elementary symmetric polynomial. In particular,
$$H:=H_1=\sigma_1(\lambda_1, \dots, \lambda_n)$$
is the mean curvature. We write
$$\stackrel{\circ}{A}:=A-\frac{H}{n}g$$
for the traceless second fundamental form of $M$, and
$$E_k:=\frac{H_k}{\binom{n}{k}}$$
for the $k$-th normalised mean curvature.

When we later specialise to the Euclidean setting, we also adopt the convention
$$E_{-1}:=\langle f, \nu\rangle, $$
which is consistent with the Minkowski formula \eqref{eq mink}.

\subsection{Oriented volume of normal graphs in space forms}

In this subsection, we derive a general formula for the oriented volume of a normal graph over a smooth closed embedded hypersurface in a space form. This formula will later be applied to focal maps and evolutes.

Let $K\in\{0, \pm1\}$, and let $S^{n+1}_K$ denote the $(n+1)$-dimensional simply connected space form of constant sectional curvature $K$. Thus
$$
S^{n+1}_1=\mathbb S^{n+1},
\qquad
S^{n+1}_0=\mathbb R^{n+1},
\qquad
S^{n+1}_{-1}=\mathbb H^{n+1}.
$$
We define
\begin{equation}\begin{split}\label{eq sk ck}
s_K(t):=
\begin{cases}
\sin t, & K=1, \\
t, & K=0, \\
\sinh t, & K=-1,
\end{cases}
\qquad
c_K(t):=
\begin{cases}
\cos t, & K=1, \\
1, & K=0, \\
\cosh t, & K=-1.
\end{cases}
\end{split}\end{equation}

These functions satisfy
$$
c_K'(t)=-K\, s_K(t),
\qquad
s_K'(t)=c_K(t),
\qquad
c_K(t)^2+K\, s_K(t)^2=1.
$$

Our ultimate interest is in the oriented volumes of focal maps or evolutes, or in differences of such oriented volumes corresponding to the largest and smallest radii of curvature. As we shall see later, these quantities provide lower or upper bounds for certain reverse isoperimetric or Minkowski-type inequalities. For this purpose, however, it is convenient to work in a more general setting: the oriented volume of a continuous normal graph over a smooth closed hypersurface in a space form. The resulting formula slightly generalises the Steiner-type formulas of Allendoerfer \cite{Allendoerfer1948}.

Let $f:M^n\to S^{n+1}_K$ be a smooth, closed, embedded hypersurface with induced measure $d\mu$ and a chosen unit normal field $\nu$. Let $\Omega$ denote the domain enclosed by $f$ whose outward unit normal is $\nu$. The higher order mean curvatures $H_k$ are defined as in the previous subsection from the principal curvatures $\lambda_1, \dots, \lambda_n$ of $M$.

We consider the normal exponential map
$$
\Phi:M\times\mathbb R\to S^{n+1}_K,
\qquad
\Phi(x, t):=\exp_{f(x)}(t\nu(x)).
$$
For $i=0, \dots, n$, define
\begin{equation}\label{eq phi_i}
\phi_i^K(s):=\int_0^s s_K(t)^{\, n-i}c_K(t)^i\, dt.
\end{equation}
\begin{definition}[Oriented volume of a normal graph]\label{def ori vol}
Let $u:M\to\mathbb R$ be a continuous function and $ \mathrm{f} (x):=\Phi(x,u(x))$. The oriented volume of $\mathrm{f}$ is defined by
$$
V_{n+1}^K[\mathrm{f}]
:=|\Omega|+\int_{M}\int_0^{u(x)} J(x,t)\,dt\,d\mu(x),
$$
where
$J(x,t)$ denotes the Jacobian of $\Phi$ with respect to the product measure $dt\,d\mu(x)$.

When $\Phi$ is one-to-one on $\Omega_u:=\{(x,t)\in M\times\mathbb R:\ t\text{ lies between }0\text{ and }u(x)\}$, the quantity $V_{n+1}^K[\mathrm{f}]-|\Omega|$ agrees with the signed volume of the region swept out by the normal segments joining $M$ to the graph of $u$.
\end{definition}

\begin{proposition}[Oriented volume formula in a space form]\label{prop oriented vol formula}
Let $u:M\to\mathbb R$ be continuous. Then for $\mathrm{f}(x)=\Phi(x,u(x))$,
\begin{equation}\label{EnclosedVolumeEqn1}
V_{n+1}^K[\mathrm{f}] = |\Omega| + \int_M\sum_{i=0}^n H_{n-i}(x)\,\phi_i^K(u(x))\,d\mu(x).
\end{equation}
In particular, in the case where $K=0$, we have
\begin{equation}\label{EnclosedVolumeEqn1'}
V_{n+1}^0[\mathrm{f}]=|\Omega|+\int_M \sum_{k=0}^n \frac{H_k(x)}{k+1} u(x)^{k+1} d \mu(x).
\end{equation}
\end{proposition}

\begin{proof}
Fix $x\in M$ and consider the normal geodesic
$$
t\mapsto \Phi(x,t)=\exp_{f(x)}(t\nu(x)).
$$
In Fermi coordinates around $M$, the ambient metric takes the form
$$
dt^2+g_{\alpha\beta}(x,t)\,dx^\alpha dx^\beta.
$$
It is by now well known, cf. \cite{Allendoerfer1948}, that the Jacobian of the normal exponential map with respect to the Fermi coordinates is
$$
J(x,t)
=
\sqrt
{\frac{{\det g_{\alpha\beta}(x,t)}}{{\det g_{\alpha\beta}(x,0)}}}
=
\prod_{j=1}^n\bigl(c_K(t)+\lambda_j(x)s_K(t)\bigr).
$$
Hence, by definition,
\begin{equation}\label{eq jac vol}
V_{n+1}^K[\mathrm{f}]-|\Omega|
=\int_M\int_0^{u(x)}J(x,t)\,dt\,d\mu(x)
=\int_M\int_0^{u(x)}
\prod_{j=1}^n\bigl(c_K(t)+\lambda_j(x)s_K(t)\bigr)\,dt\,d\mu(x).
\end{equation}

Expanding the product gives
$$
\prod_{j=1}^n\bigl(c_K(t)+\lambda_j(x)s_K(t)\bigr)
=
\sum_{i=0}^n H_{n-i}(x)\,s_K(t)^{\,n-i}c_K(t)^i.
$$
Substituting this into the preceding identity and using the definition of $\phi_i^K$, we obtain
$$
V_{n+1}^K[\mathrm{f}]-|\Omega|
=
\int_M\sum_{i=0}^n H_{n-i}(x)\,\phi_i^K(u(x))\,d\mu(x).
$$
This proves the formula.
\end{proof}
\section{Reverse Alexandrov-Fenchel inequalities in Euclidean space}\label{SectionEuclideanSpace}

\subsection{Support-function formulae}\label{SupportFunctionConvexSubsection1}

We now restrict to the Euclidean setting and assume that \(f:M^n\to \mathbb R^{n+1}\) is a smooth, closed, strictly convex hypersurface. Since \(f\) is strictly convex, its Gauss map is a diffeomorphism. After composing \(f\) with the inverse Gauss map, we may therefore regard \(f\) as a smooth map
\[
f:\mathbb S^n\to\mathbb R^{n+1},
\qquad
\nu(z)=z.
\]
Let \(\sigma\), \(\overline{\nabla}\), \(\overline{\Delta}\), and \(d\sigma\) denote the round metric, Levi-Civita connection, Laplace--Beltrami operator, and volume element on \(\mathbb S^n\), respectively.

Define the support function
by
\[
h(z):=\inner{f(z),z}.
\]
Differentiating gives the standard representation
\[
f=hz+\overline{\nabla}h.
\]

Define the radius-of-curvature tensor
\[
r_{ij}:=\overline{\nabla}_{i}\overline{\nabla}_{j}h+h\,\sigma_{ij}.
\]
Then
\[
f_{i}=\sigma^{jk}r_{ij}z_{k},
\qquad
g_{ij}=r_{ik}\sigma^{kl}r_{lj}.
\]
Thus the endomorphism \(r_{i}{}^{j}:=\sigma^{jk}r_{ik}\) is the inverse
Weingarten map. Its eigenvalues are the principal radii
\[
\rho_{i}:=\lambda_{i}^{-1}.
\]
Consequently,
\begin{equation}\label{EuclideanSupportFunctionEqn1}
\begin{aligned}
E_n=\frac{1}{\det(r_i^{\,j})},\qquad E_{n-1}=\frac{\operatorname{tr}_\sigma r}{n \det(r_i^{\,j})}=\frac{1}{n}(\overline \Delta h+nh)E_n,\qquad E_nd\mu=d\sigma.
\end{aligned}
\end{equation}

\begin{proposition}\label{prop identities}
Let $\langle \cdot,\cdot\rangle_{\sigma}$ denote the $L^2$ inner product on $\mathbb S^n$. Let $h:\mathbb S^{n}\to \mathbb R$ be the support function of a strictly convex hypersurface $M\subset \mathbb R^{n+1}$. Then the following identities hold:
\begin{enumerate}
\item\label{prop id 1}
$\int_{M} E_{n-1} d\mu=\int_{\mathbb{S}^n} h d \sigma$
\item\label{prop id 2}
$\int_{M} E_{n-2} d\mu=\frac{1}{n}\langle h, \overline{\Delta} h+n h\rangle_{\sigma}$
\item\label{prop id 3}
If $u=h-\overline h$, where
$\overline h=\frac{1}{|\mathbb S^n|}\int_{\mathbb S^n}h\, d\sigma$ is the average of $h$ over $\mathbb S^n$,
then
$$
\frac{1}{|\mathbb S^n|}\left(\int_{M} E_{n-1}\, d\mu\right)^2-\int_{M} E_{n-2}\, d\mu
=\frac{1}{n}\langle u, -\overline\Delta u-n u\rangle_{\sigma}.
$$
\end{enumerate}
\end{proposition}
\begin{proof}
By \eqref{EuclideanSupportFunctionEqn1},
$\frac{E_{n-1}}{E_{n}}=h+\frac{1}{n}\overline\Delta h$ and $E_{n}\, d\mu=d\sigma$.
It follows that
$$\int_M E_{n-1}\, d\mu=\int_{\mathbb S^n} h\, d\sigma. $$
This proves \eqref{prop id 1}.

By the Minkowski formula (which also holds for $n=1$ trivially),
\begin{equation}\label{eq mink}
\int_M E_{n-2}\, d\mu=\int_M E_{n-1}\langle f, \nu\rangle\, d\mu,
\end{equation}
we obtain
$$
\int_M E_{n-2}\, d\mu
=\int_{\mathbb S^n} h\, \frac{E_{n-1}}{E_{n}}\, d\sigma
=\int_{\mathbb S^n} h\left(h+\frac{1}{n}\overline\Delta h\right)\, d\sigma
=\frac{1}{n}\langle h, \overline\Delta h+nh\rangle_{\sigma},
$$
which proves \eqref{prop id 2}.

For \eqref{prop id 3}, first of all, by \eqref{prop id 1}, we have
$$
\overline h=\frac{1}{|\mathbb S^n|}\int_{\mathbb S^n}h\, d\sigma
=\frac{1}{|\mathbb S^n|}\int_M E_{n-1}\, d\mu.
$$
On the other hand, since $h=u+\overline h$ and $\int_{\mathbb S^n}u\, d\sigma=0$, we get
\begin{align*}
n\int_{M} E_{n-2} d \mu=\langle h, \overline\Delta h+n h\rangle_{\sigma}
=& \langle u, \overline\Delta u+n u\rangle_{\sigma}+n|\mathbb S^n|\overline h^{2}\\
=& \langle u, \overline\Delta u +n u \rangle_{\sigma} +\frac{n}{\left|\mathbb{S}^n\right|}\left(\int_{M} E_{n-1} d \mu\right)^2.
\end{align*}
This proves \eqref{prop id 3}.
\end{proof}

We also need the following algebraic lemmas.
\begin{lemma}\label{lem alg id}
Suppose $n\ge2$ and $\lambda_1, \ldots, \lambda_n>0$, and let $B=\operatorname{diag}\left(\frac{1}{\lambda_1}, \ldots, \frac{1}{\lambda_n}\right)$.
For $0 \le k \le n$, let $\sigma_k=\sigma_k\left(\lambda_1, \ldots, \lambda_n\right)$ be the $k$-th elementary symmetric polynomial, and let
$E_k:=\frac{\sigma_k}{\binom{n}{k}}$. Denote by $\stackrel {\circ}B:=B-\frac{\operatorname{tr}(B)}{n} I$
the traceless part of $B$. Then
$$
|\stackrel {\circ}B|^2
=\sum_{i=1}^n\left(\frac{1}{\lambda_i}-\frac{1}{n} \sum_{j=1}^n \frac{1}{\lambda_j}\right)^2
=n(n-1) \frac{E_{n-1}^2-E_{n-2} E_n}{E_n^2}.
$$
\end{lemma}
\begin{proof}
Since $B=\operatorname{diag}(1/\lambda_1, \dots, 1/\lambda_n)$, we have
$$
\operatorname{tr}(B)=\sum_{i=1}^n \frac{1}{\lambda_i},
\qquad
|B|^2=\sum_{i=1}^n \frac{1}{\lambda_i^2}.
$$
Hence
$$
|\stackrel \circ B|^2=|B|^2-\frac{1}{n}(\operatorname{tr}B)^2
=\sum_{i=1}^n \frac{1}{\lambda_i^2}-\frac{1}{n}\left(\sum_{i=1}^n \frac{1}{\lambda_i}\right)^2.
$$
Now
$$\sum_{i=1}^n \frac{1}{\lambda_i}=\frac{\sigma_{n-1}}{\sigma_n}, $$
since $\sigma_{n-1}=\sigma_n\sum_i \lambda_i^{-1}$, and also
$$
\sum_{i=1}^n \frac{1}{\lambda_i^2}
=\left(\sum_{i=1}^n \frac{1}{\lambda_i}\right)^2-2\sum_{i<j}\frac{1}{\lambda_i\lambda_j}
=\frac{\sigma_{n-1}^2-2\sigma_{n-2}\sigma_n}{\sigma_n^2},
$$
since $\sigma_{n-2}=\sigma_n\sum_{i<j}(\lambda_i\lambda_j)^{-1}$. Substituting these identities into the formula for $|\stackrel \circ B|^2$, we obtain
$$
|\stackrel \circ B|^2
=\frac{\sigma_{n-1}^2-2\sigma_{n-2}\sigma_n}{\sigma_n^2}
-\frac{1}{n}\left(\frac{\sigma_{n-1}}{\sigma_n}\right)^2
=\frac{(n-1)\sigma_{n-1}^2-2n\sigma_{n-2}\sigma_n}{n\sigma_n^2},
$$
and the result follows.
\end{proof}

\begin{lemma}\label{lem:newton-deficit-rho-gap}
Let $n\ge2$, $\lambda_1\ge\lambda_2\ge\cdots\ge\lambda_n>0$, and $E_k:=\frac{\sigma_k}{\binom{n}{k}}$ be the normalised elementary symmetric polynomials in $\lambda_1, \dots, \lambda_n$. Set $\rho_i:=1/\lambda_i$. Then
\begin{equation}\label{ineq maclaurin rho}
\frac{E_n}{2n(n-1)}(\rho_n-\rho_1)^2 \le \frac{E_{n-1}^2-E_{n-2}E_n}{E_n} \le \frac{E_n}{2n}(\rho_n-\rho_1)^2.
\end{equation}
If $n=2$, this is an identity.
Suppose $n\ge 3$, then the equality on the right holds if and only if all $\rho_i$ are equal, and the equality on the left holds if and only if $\rho_2=\cdots=\rho_{n-1}=\frac{1}{2}\left(\rho_1+\rho_n\right)$.

For $n\ge3$, we also have
\begin{equation}\label{ineq maclaurin traceless A}
E_{n-1}^2-E_{n-2} E_n \le \frac{1}{n^2(n-1)}\left(\frac{n}{n-2}\right)^{2 n-4} E_1^{2 n-4} \sum_{i<j}\left(\lambda_i-\lambda_j\right)^2.
\end{equation}
The equality holds if and only if all $\lambda_i$ are equal.
\end{lemma}

\begin{proof}
Since $\rho_i=\frac{1}{\lambda_i}$, we have
$$
E_{n-1}
=
E_n\cdot \frac{1}{n}\sum_{i=1}^n \rho_i,
\qquad
E_{n-2}
=
E_n\cdot \frac{1}{\binom{n}{2}}\sum_{i<j}\rho_i\rho_j.
$$
Therefore
$$
\frac{E_{n-1}^2-E_{n-2}E_n}{E_n}
=
E_n\left[
\left(\frac{1}{n}\sum_{i=1}^n \rho_i\right)^2
-
\frac{1}{\binom{n}{2}}\sum_{i<j}\rho_i\rho_j
\right].
$$
Using the identity
$$
\left(\frac{1}{n}\sum_{i=1}^n \rho_i\right)^2
-
\frac{1}{\binom{n}{2}}\sum_{i<j}\rho_i\rho_j
=
\frac{1}{n^2(n-1)}\sum_{i<j}(\rho_i-\rho_j)^2,
$$
we obtain
\begin{equation}\label{eq diff maclaurin}
\frac{E_{n-1}^2-E_{n-2}E_n}{E_n} = \frac{E_n}{n^2(n-1)}\sum_{i<j}(\rho_i-\rho_j)^2.
\end{equation}

The inequality on the right of \eqref{ineq maclaurin rho} then follows from the estimate
$\sum_{i<j}(\rho_i-\rho_j)^2 \le \binom{n}{2}\left(\rho_n-\rho_1\right)^2$, and the equality holds if and only if $n=2$ or $\rho_1=\rho_n$.

To get the lower bound, we use the identity
$\sum_{i<j}\left(\rho_j-\rho_i\right)^2=n \sum_{i=1}^n\left(\rho_i-\bar{\rho}\right)^2$, where $\bar{\rho}:=\frac{1}{n} \sum_{i=1}^n \rho_i$. From this,
\begin{equation}\label{var to maxmin}
\begin{aligned}
\sum_{i<j}\left(\rho_j-\rho_i\right)^2=n \sum_{i=1}^n\left(\rho_i-\bar{\rho}\right)^2
\ge& n\left(\left(\rho_1-\bar{\rho}\right)^2+\left(\rho_n-\bar{\rho}\right)^2\right)\\
=& n\left(\left(\bar{\rho}-\rho_1\right)^2+\left(\rho_n-\bar{\rho}\right)^2\right)\\
\ge& \frac n2 \left(\rho_n-\rho_1\right)^2,
\end{aligned}
\end{equation}
where we have used $a^2+b^2 \ge\frac{(a+b)^2}{2}$ in the last line. Putting this into \eqref{eq diff maclaurin} gives the inequality on the left of \eqref{ineq maclaurin rho}.
{{The equality holds if and only if $\bar \rho=\rho_2=\cdots=\rho_{n-1}=\frac{1}{2}\left(\rho_1+\rho_n\right)$.}}

To prove \eqref{ineq maclaurin traceless A} for $n\ge 3$, we rewrite \eqref{eq diff maclaurin} as
\begin{equation}\label{ineq maclaurin traceless A'}
\begin{aligned}
E_{n-1}^2-E_{n-2} E_n=&\frac{E_n^2}{n^2(n-1)} \sum_{1 \leq i<j \leq n}\left(\frac{1}{\lambda_i}-\frac{1}{\lambda_j}\right)^2\\
=&\frac{1}{n^2(n-1)} \sum_{1\le i<j\le n}\left(\prod_{k \neq i, j} \lambda_k\right)^2\left(\lambda_i-\lambda_j\right)^2.
\end{aligned}
\end{equation}

By the AM-GM inequality, we have
\begin{equation}\label{ineq am gm}
\prod_{k \neq i, j} \lambda_k \leq\left(\frac{\sum_{k \neq i, j} \lambda_k}{n-2}\right)^{n-2} <\left(\frac{n}{n-2} E_1\right)^{n-2}
\end{equation}
and so
$$
E_{n-1}^2-E_{n-2} E_n \leq \frac{1}{n^2(n-1)}\left(\frac{n}{n-2}\right)^{2 n-4} E_1^{2 n-4} \sum_{i<j}\left(\lambda_i-\lambda_j\right)^2.
$$
The inequality \eqref{ineq am gm} is strict. Hence, by \eqref{ineq maclaurin traceless A'}, the above inequality is strict unless all the $\lambda_i$ are equal, in which case equality holds.
\end{proof}

\subsection{Focal maps and oriented volume}

The focal objects are most naturally treated as maps rather than as point sets.
Let us define the ordered principal radius functions
\[
\rho_{i}:=\lambda_{i}^{-1}.
\]
We assume that the $\rho_i$ are arranged in {{ascending}} order
\begin{equation}
\rho_1\le\cdots\le \rho_n.
\end{equation}
As the ordered principal curvatures are the ordered eigenvalues of the shape operator, viewed locally as a smooth symmetric endomorphism, Weyl's eigenvalue inequality implies that the functions $\rho_i$ are locally Lipschitz.

We then define the focal maps
\[
b_{i}:M\to\mathbb R^{n+1},
\qquad
b_{i}:=f-\rho_{i}\nu,
\qquad i=1,\cdots,n.
\]
They need not be immersions, but their oriented volumes (Definition \ref{def ori vol}), which we simply denote by \(V[b_{i}]\), are
well-defined, and are given by \eqref{EnclosedVolumeEqn1'}.
Later on, we are going to estimate $|V[b_n]-V[b_1]|$. In view of \eqref{EnclosedVolumeEqn1'}, it is natural to consider the following estimate
\begin{lemma}\label{lem:F-rho-gap}
Let $n\ge2$, $\lambda_1\ge \dots \ge\lambda_n>0$, and set
$\rho_i:=\frac{1}{\lambda_i}$.
Let $\sigma_k$ denote the $k$ th elementary symmetric polynomial in $\lambda_1, \ldots, \lambda_n$, $E_n=\lambda_1 \dots \lambda_n$, and define
$F(u):=\sum_{k=0}^n \frac{\sigma_k}{k+1}u^{k+1}$.
Then
$$
\left|F(-\rho_n)-F(-\rho_1)\right|
\le \frac{E_n}{{{n(n+1)}}}(\rho_n-\rho_1)^{n+1}.
$$
If $n=2$, this is an equality.
If $n\ge 3$, then the equality holds if and only if either all $\lambda_i$ are equal, or $\lambda_1=\lambda_2=\cdots=\lambda_{n-1}>\lambda_n$, or $\lambda_1>\lambda_2=\cdots=\lambda_n$.
\end{lemma}

\begin{proof}
Differentiating $F$ gives
$$
F'(u)=\sum_{k=0}^n \sigma_k u^k=\prod_{i=1}^n(1+\lambda_i u)
=E_n\prod_{i=1}^n(u+\rho_i).
$$
Hence
\begin{align*}
F(-\rho_n)-F(-\rho_1)
= \int^{-\rho_n}_{-\rho_1}F'(u)\, du
=& E_n \int^{-\rho_n}_{-\rho_1}\prod_{i=1}^n(u+\rho_i)\, du \\
=&(-1)^{n+1} E_n \int_{\rho_1}^{\rho_n} \prod_{i=1}^n\left(\tau-\rho_i\right) d \tau.
\end{align*}

Suppose $\rho_n=\rho_1$, then there is nothing to prove. Otherwise, we make the substitution $\tau=\rho_1+\left(\rho_n-\rho_1\right) x$ for $x \in[0,1]$ and set $\alpha_i:=\frac{\rho_i-\rho_1}{\rho_n-\rho_1}$. Then
\begin{equation*}\label{F diff}
\begin{split}
F(-\rho_n)-F(-\rho_1)
=&(-1)^{n+1}\left(\rho_n-\rho_1\right) E_n \int_0^1 \prod_{i=1}^n\left(\left(\rho_1-\rho_i\right)+\left(\rho_n-\rho_1\right) x\right) d x\\
=&(-1)^{n+1}\left(\rho_n-\rho_1\right)^{n+1} E_n \int_0^1 \prod_{i=1}^n\left(x-\alpha_i\right) d x.
\end{split}
\end{equation*}
Therefore,
\begin{align}\label{ineq F diff'}
\left|F\left(-\rho_n\right)-F\left(-\rho_1\right)\right| \leq\left(\rho_n-\rho_1\right)^{n+1} E_n\int_0^1 \prod_{i=1}^n |x-\alpha_i| d x.
\end{align}
We now estimate the integral $I\left(\alpha_2, \ldots, \alpha_{n-1}\right):=\int_0^1 \prod_{i=1}^n |x-\alpha_i| d x=\int_0^1 x(1-x) \prod_{i=2}^{n-1}\left|x-\alpha_i\right| d x$.
Observe that $I(\alpha_2,\cdots,\alpha_{n-1})$ is convex in each $\alpha_i$ and by construction, $0=\alpha_1 \leq \alpha_2 \leq \cdots \leq \alpha_{n-1} \leq \alpha_n=1$. Therefore,
\begin{align*}
I(\alpha_2,\cdots,\alpha_{n-1})
\le& \max \{I(0, \alpha_3,\cdots,\alpha_{n-1}) ,I(1, \alpha_3,\cdots,\alpha_{n-1})\}\\
=& \max \left\{\int_0^1 x^2(1-x) \prod_{i=3}^{n-1}\left|x-\alpha_i\right| d x,
\int_0^1 x(1-x)^2 \prod_{i=3}^{n-1}\left|x-\alpha_i\right| d x \right\}.
\end{align*}
Inductively, we then have
\begin{equation}\label{ineq I}
I (\alpha_2,\cdots,\alpha_{n-1})\le \max _{0 \leq l \leq n-2} \int_0^1 x^{l+1}(1-x)^{n-(l+1)} d x.
\end{equation}
We estimate
\begin{equation}\label{ineq beta}
\begin{aligned}
\max_{0\le l\le n-2}\int_0^1 x^{l+1}(1-x)^{n-(l+1)}dx
=& \max_{0\le l\le n-2} B(l+2,n-l)\\
=& \max_{0\le l\le n-2}\frac{(l+1)!(n-(l+1))!}{(n+1)!}\\
=& \frac{1}{n+1}\max_{0\le l\le n-2}\frac{1}{\binom{n}{l+1}}\\
=& \frac{1}{n(n+1)}.
\end{aligned}
\end{equation}

Here $B(a, b)=\int_0^1 x^{a-1}(1-x)^{b-1} d x$ is the Beta function. Note that $\max _{0 \leq l \leq n-2} \frac{1}{\binom{n}{l+1}}=\frac{1}{n}$ occurs when $l=0, n-2$.
Therefore by \eqref{ineq F diff'} we have
\begin{equation}\label{ineq F-rho-gap}
\left|F\left(-\rho_n\right)-F\left(-\rho_1\right)\right| \leq \frac{E_n}{n(n+1)}\left(\rho_n-\rho_1\right)^{n+1}
\end{equation}
as required.

It is easy to see that when $n=2$, \eqref{ineq F diff'} and \eqref{ineq I} are equalities; hence \eqref{ineq F-rho-gap} is an equality. If $n\ge3$ and the equality in \eqref{ineq F-rho-gap} holds, then we have either the trivial case where $\rho_1=\rho_n$, or $\rho_1<\rho_n$. In the latter case, \eqref{ineq I} and \eqref{ineq beta} then implies either $0=\alpha_1=\cdots =\alpha_{n-1}$ or $\alpha_2=\cdots =\alpha_{n}=1$, or in other words, $\lambda_1=\lambda_2=\cdots=\lambda_{n-1}>\lambda_n$ or $\lambda_1>\lambda_2=\cdots=\lambda_n$.

\end{proof}

\subsection{Main results}
\begin{theorem}\label{thm higher dim reverse mink}
Let $M$ be a smooth closed strictly convex hypersurface in $\mathbb R^{n+1}$, and $E_k$ be the normalised $k$-th mean curvature of $M$. Then
\begin{equation}\label{ineq higher dim reverse mink}
0\le\frac{1}{\left|\mathbb{S}^n\right|}\left(\int_{M} E_{n-1} d \mu\right)^2-\int_{M} E_{n-2} d \mu
\le \frac{n}{2(n+1)} \int_M \frac{E_{n-1}^2-E_{n-2} E_n}{E_n} d\mu.
\end{equation}
\end{theorem}
\begin{remark}
\begin{enumerate}
\item
Note that, when $n\ge2$, the integrand $\frac{1}{E_n}(E_{n-1}^2-E_{n-2}E_n)$ on the right-hand side of \eqref{ineq higher dim reverse mink} is non-negative by the Newton--Maclaurin inequality. When $n=1$, the integral on the  right-hand side is $\int_M(\frac{1}{\kappa}-\langle f,\nu\rangle)\,ds=\int_M\frac{1}{\kappa}\,ds-2|\Omega|$, where $\Omega$ is the domain enclosed by $M$; this is non-negative by the Heintze--Karcher inequality.
\item
When $n=1$, this inequality reads as
$$
\frac{2}{\pi}\left(L^2-4 \pi |\Omega|\right)\le \int_{M} \frac{1}{\kappa} d s-2|\Omega|.
$$
This was proved by Lin and Tsai \cite[Lemma 1.7]{LinTsai2012}. See also \cite{KwongLee2021} for further generalisations. It can also be proved by combining Hurwitz's inequality \eqref{HurwitzInequality} and \cite[Theorem 1]{EscuderoReventos2007}.
\item When $n=2$, this inequality is reduced to
$$0 \le \left(\int_M H d \mu\right)^2-16 \pi|M| \le \frac{8 \pi}{3} \int_M \frac{|\stackrel \circ A|^2}{{\Gauss}} d \mu. $$
\end{enumerate}
\end{remark}

\begin{proof}[Proof of Theorem \ref{thm higher dim reverse mink}]
Let $u=h-\overline h$.
By Proposition \ref{prop identities} and Lemma \ref{lem wirtinger},
\begin{equation}\label{ineq mink1}
\begin{aligned}
0\le\frac{1}{\left|\mathbb{S}^n\right|}\left(\int_{M} E_{n-1} d \mu\right)^2-\int_{M} E_{n-2} d \mu
=& \frac{1}{n}\langle u, -\overline{\Delta} u-n u\rangle_\sigma\\
\le& \frac{1}{2(n+1)}\left(\frac{1}{n} \int_{\mathbb{S}^n}(\overline{\Delta} u)^2 d \sigma-\int_{\mathbb{S}^n}|\overline{\nabla} u|^2 d \sigma\right)\\
=& \frac{1}{2(n+1)}\left(\frac{1}{n} \int_{\mathbb{S}^n}(\overline{\Delta} h)^2 d \sigma-\int_{\mathbb{S}^n}|\overline{\nabla} h|^2 d \sigma\right).
\end{aligned}
\end{equation}

Suppose $n\ge2$ at the moment. Integrating the Bochner formula gives
$$
\int_{\mathbb{S}^n}(\overline{\Delta} h)^2 d \sigma=\int_{\mathbb{S}^n}\left|\overline{\nabla}^2 h\right|^2 d \sigma+(n-1)\int_{\mathbb{S}^n}|\overline{\nabla} h|^2 d \sigma.
$$
Hence
\begin{equation}\label{ineq mink2}
\frac{1}{n} \int_{\mathbb{S}^n}(\overline{\Delta} h)^2 d \sigma-\int_{\mathbb{S}^n}|\overline{\nabla} h|^2 d \sigma=\frac{1}{n-1}\int_{\mathbb{S}^n}\left(\left|\overline{\nabla}^2 h\right|^2-\frac{1}{n}(\overline{\Delta} h)^2\right) d \sigma.
\end{equation}
Now let
$$
\stackrel {\circ}r_{i j}:=r_{i j}-\frac{1}{n}\left(\operatorname{tr}_\sigma r\right) \sigma_{i j}=\overline{\nabla}_i \overline{\nabla}_j h-\frac{1}{n}(\overline{\Delta} h) \sigma_{i j},
$$
which is the traceless part of $r=\overline \nabla^2 h + h \overline g$, the inverse Weingarten tensor, so that
$$ \left|\overline{\nabla}^2 h\right|^2-\frac{1}{n}(\overline{\Delta} h)^2=|\stackrel {\circ}r|^2. $$

Since $r_i{ }^j$ has eigenvalues $1/\lambda_i$, where $\lambda_i$'s are the principal curvatures, by Lemma \ref{lem alg id}, we have that
$$|\stackrel \circ r|^2
=n(n-1) \frac{E_{n-1}^2-E_{n-2} E_n}{E_n^2}.
$$
In view of \eqref{ineq mink1}, \eqref{ineq mink2} and using $d\sigma=E_{n}d\mu$, the result follows.

Now suppose $n=1$. Then $L:=\int_M E_0d\mu=|M|$ is the length and $A:=\frac{1}{2}\int_M E_{-1}d\mu=\frac{1}{2}\int_M\langle f,\nu\rangle d\mu$ is the area enclosed by $M$. Using \eqref{ineq mink1}, and recalling that $h''+h=1/\kappa$ and $d\mu=(h''+h)\,d\sigma$, we obtain
\begin{align*}
0
\le 4\left(\frac{L^2}{2\pi}-2A\right)
\le \int_{\mathbb S^1} (h'')^2\,d\sigma - \int_{\mathbb S^1} (h')^2\,d\sigma
&= \int_{\mathbb S^1} (h'')^2\,d\sigma + \int_{\mathbb S^1} h h''\,d\sigma \\
&= \int_{\mathbb S^1} (h''+h)^2\,d\sigma - \int_{\mathbb S^1} h(h''+h)\,d\sigma \\
&= \int_M \frac{1}{\kappa}\,d\mu - \int_M \langle f,\nu\rangle\,d\mu \\
&= \int_M \frac{1}{\kappa}\,d\mu - 2A .
\end{align*}
Rearranging, the result follows.
\end{proof}
\begin{remark}
The second inequality in Theorem \ref{thm higher dim reverse mink} admits non-round equality cases. It is easy to see that equality can hold only if the support function has no spherical harmonic components of degree $k\geq 3$. However, it seems to be a more difficult problem to find a simple algebraic condition on the spherical harmonic coefficients that is necessary and sufficient for the hypersurface to be strictly convex. For related discussion in the planar curve case, see \cite[Remark 1.9]{LinTsai2012} and \cite[Remark 3.1]{KwongLee2021}. Let us give one example in the curve case and one example in the surface case.
\begin{enumerate}
\item
When $n=1$, we identify $\mathbb{R}^2$ with $\mathbb{C}$. Take $a_0=1$, $a_2=\frac{1}{10}$, and $a_1=0$, with $a_{-n}=\overline{a_n}$. Let $h(\theta)=\sum_{n=-2}^2 a_n e^{in\theta}=1+\frac{1}{5}\cos 2\theta$. Then the curve
$$
\gamma(\theta)=\sum_{n=-2}^2 a_n(1-n)e^{i(n+1)\theta}
=e^{i\theta}+\frac{3}{10}e^{-i\theta}-\frac{1}{10}e^{3i\theta}.
$$
has support function $h$.
Moreover,
\[
h(\theta)+h''(\theta)=1-\frac{3}{5}\cos 2\theta>0
\]
for all \(\theta\). Since
\[
\gamma(\theta)=(h(\theta)+i h'(\theta))e^{i\theta},
\qquad
\gamma'(\theta)=i\bigl(h(\theta)+h''(\theta)\bigr)e^{i\theta},
\]
the tangent angle is \(\theta+\pi/2\), and the curvature is
\[
\kappa(\theta)=\frac{1}{h(\theta)+h''(\theta)}>0.
\]
Thus \(\gamma\) is the standard support-function parametrisation of a smooth
strictly convex oval; in particular, it is embedded.
Hence \(\gamma\) is a closed embedded strictly convex curve in
\(\mathbb R^2\) which is non-round, and equality holds in this case.

\item
Now, let $n=2$. Let $h(\phi,\theta)=1+\frac{\varepsilon}{2}(3\cos^2\phi-1)$, where $\varepsilon=\frac{2}{5}$ and $(\phi,\theta)\in[0,\pi]\times[0,2\pi)$. Here $(\phi,\theta)$ are the standard spherical coordinates on $\mathbb S^2$, so that the unit sphere is parametrised by $\nu(\phi,\theta)=(\sin\phi\cos\theta,\sin\phi\sin\theta,\cos\phi)$. Notice that $3\cos^2\phi-1$ is a second spherical harmonic. Equivalently,
$h(\phi,\theta)=1+\varepsilon P_2(\cos \phi)$, where $P_2(t)=\frac{1}{2}\left(3 t^2-1\right)$ is the second Legendre polynomial (cf. \cite[Lemma 3.1.3 and p. 85]{groemer1996geometric}).

The surface with support function $h$ is given by $f(\phi,\theta)=h(\phi,\theta)\nu(\phi,\theta)+\overline \nabla h(\phi,\theta)$. In coordinates,
$$f(\phi,\theta)
=
\begin{pmatrix}
\sin\phi\left(1-\frac{\varepsilon}{2}(1+3\cos^2\phi)\right)\cos\theta \\
\sin\phi\left(1-\frac{\varepsilon}{2}(1+3\cos^2\phi)\right)\sin\theta \\
\cos\phi\left(1+\frac{\varepsilon}{2}(5-3\cos^2\phi)\right)
\end{pmatrix}.
$$
The eigenvalues of $\overline \nabla^2 h+h\sigma$ with respect to $\sigma$ are $1+\frac{\varepsilon}{2}(5-9\cos^2\phi)$ and $1-\frac{\varepsilon}{2}-\frac{3\varepsilon}{2}\cos^2\phi$. Both are bounded below by $1-2\varepsilon=\frac{1}{5}>0$. Hence $\overline \nabla^2 h+h\sigma>0$, and $f$ is a strictly convex embedded hypersurface by the Hadamard convexity theorem \cite[Theorem 5.17]{GallotHulinLafontaine2004}.
So $f$ parametrises a non-round strictly convex surface in $\mathbb R^3$. Since the non-constant part of $h$ is a second spherical harmonic, it attains equality in Theorem \ref{thm higher dim reverse mink}. This non-round surface is shown below.
\begin{figure}[htbp]
\centering
\includegraphics[width=0.18\textwidth]{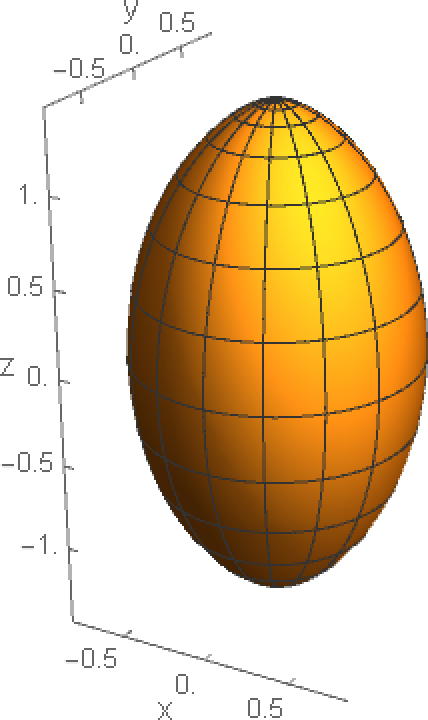}
\label{fig:nonround-surface}
\end{figure}
\end{enumerate}
\end{remark}

Recall that we define the focal maps
$$
b_i: M \rightarrow \mathbb{R}^{n+1}, \quad b_i:=f-\rho_i \nu, \quad i=1,\cdots,n.
$$
Here, $\rho_i=\frac{1}{\lambda_i}$ and we assume that $0<\rho_1\le \cdots\le \rho_n$.
The maps $b_i$ are continuous, and hence their oriented volumes, as defined in Definition~\ref{def ori vol}, are well defined. Moreover, they are given by \eqref{EnclosedVolumeEqn1'}. For simplicity, in this subsection, we write $V[b_i]$ instead of $V_{n+1}^0[b_i]$.

\begin{theorem}\label{thm euclidean 2}
Let $M$ be a smooth closed strictly convex hypersurface in $\mathbb{R}^{n+1}$ ($n\ge2$), $E_k$ be the normalised $k$-th mean curvature of $M$, and $\rho_1$ and $\rho_n$ be the smallest and largest radius of curvature respectively. Then the following inequalities hold.
\begin{enumerate}
\item
\begin{equation}\label{ineq 3}
{{\frac{n(n+1)}{(2 n(n-1))^{\frac{n+1}{2}}}}}\left|V\left[b_n\right]-V\left[b_1\right]\right| \le \int_M \frac{\left(E_{n-1}^2-E_{n-2} E_n\right)^{\frac{n+1}{2}}}{E_n^n} d\mu.
\end{equation}
This is an equality when $n=2$. If $n\ge3$, then the equality holds if and only if $M$ is a hypersphere.
\item
We also have
\begin{equation}\label{ineq 4}
0 \le \frac{1}{\left|\mathbb{S}^n\right|}\left(\int_M E_{n-1} d \mu\right)^2-\int_M E_{n-2} d \mu \le \frac{1}{4(n+1)} \int_M E_n\left(\rho_n-\rho_1\right)^2 d \mu.
\end{equation}
If $n\ge 3$, the equality in the second inequality holds if and only if $M$ is a hypersphere.
\item
For $n\ge3$, we have
\begin{equation}\label{ineq 5}
\begin{split}
0 \leq& \frac{1}{|\mathbb S^n|}\left(\int_M E_{n-1}\,d\mu\right)^2-\int_M E_{n-2}\,d\mu \\
\leq& \frac{1}{2(n+1)(n-1)}\left(\frac{n}{n-2}\right)^{2n-4}
\int_M \frac{E_1^{2n-4}}{E_n}|\stackrel \circ A|^2\,d\mu.
\end{split}
\end{equation}
The equality in the second inequality holds if and only if $M$ is a hypersphere.
\end{enumerate}
\end{theorem}
\begin{proof}
Let $F(u):=\sum_{k=0}^n \frac{H_k}{k+1} u^{k+1}$. Using \eqref{EnclosedVolumeEqn1'}, we compute
\begin{equation}
\begin{aligned}
\left|V\left[b_n\right]-V\left[b_1\right]\right|
=& \left|\int_M\left[F\left(-\rho_n\right)-F\left(-\rho_1\right)\right] d \mu\right| \\
\le& \int_M\left|F\left(-\rho_n\right)-F\left(-\rho_1\right)\right| d \mu \\
\le& \int_M\frac{E_n}{{{n(n+1)}}}\left(\rho_n-\rho_1\right)^{n+1} d \mu \\
\le& {{\frac{(2 n(n-1))^{\frac{n+1}{2}}}{n(n+1)}}}\int_M \frac{\left(E_{n-1}^2-E_{n-2} E_n\right)^{\frac{n+1}{2}}}{E_n^n} d \mu.
\end{aligned}
\end{equation}
where we have used Lemma \ref{lem:F-rho-gap} and the first inequality in \eqref{ineq maclaurin rho} in the last two lines. This proves \eqref{ineq 3}.
The equality holds for $n=2$. For $n\geq 3$, equality holds if and only if $M$ is totally umbilical, by combining the equality cases in Lemma \ref{lem:F-rho-gap} and the first inequality in \eqref{ineq maclaurin rho}; hence $M$ is a hypersphere.

The inequality \eqref{ineq 4} follows from Theorem \ref{thm higher dim reverse mink} and inequality \eqref{ineq maclaurin rho}.

The inequality \eqref{ineq 5} and its equality case follows from Theorem \ref{thm higher dim reverse mink}, inequality \eqref{ineq maclaurin traceless A} and the identity $\displaystyle|\stackrel\circ A|^2=\frac1n\sum_{i<j}\left(\lambda_i-\lambda_j\right)^2$.

Let us now discuss the equality in the second inequality of \eqref{ineq 4}.
By the equality case in the spectral inequality used in the proof of Theorem \ref{thm higher dim reverse mink}, the support function $h$ contains no spherical harmonic components of degree $k\ge 3$. Lemma \ref{lem:newton-deficit-rho-gap} does not impose further restriction on $h$ when $n=2$.

Now suppose $n\ge 3$. Since $h$ has no spherical harmonic components of degree $k\ge 3$, it is the restriction to $\mathbb S^n$ of the following function:
$$ h(x) = c+\sum_{i=1}^{n+1}b_i x_i+\sum_{i, j=1}^{n+1}p_{ij}x_i x_j =:h_0(x)+h_1(x)+h_2(x), $$
where $P=(p_{ij})$ is a symmetric trace-free $(n+1)\times(n+1)$ matrix. Indeed, $h_2$ is the restriction to $\mathbb S^n$ of a harmonic quadratic polynomial in $\mathbb R^{n+1}$ and so $P$ is trace-free (\cite[Lemma 3.13]{groemer1996geometric}).

We claim that $P=0$.

We have $\overline \nabla^2 h_1 + h_1\sigma =0$. For the quadratic part, a direct computation gives
$$
\overline \nabla^2 h_2+h_2\sigma
=
2P|_{x^\perp}-h_2(x)\sigma,
$$
where $P|_{x^\perp}$ denotes the restriction of the bilinear form $P$ to
$T_x\mathbb S^n=x^\perp$. Therefore
$$
\overline \nabla^2 h+h\sigma
=
\left(c-h_2(x)\right)\sigma+2P|_{x^\perp}.
$$
Thus, if $\mu_1(x)\le \cdots\le \mu_n(x)$ are the eigenvalues of $P|_{x^\perp}$, then the eigenvalues
$\rho_1(x)\le \cdots\le \rho_n(x)$ of $\nabla^2h+h\sigma$ with respect to $\sigma$ are
$$
\rho_i(x)=c-h_2(x)+2\mu_i(x),
\qquad i=1, \ldots, n.
$$
On the other hand, the equality condition from Lemma \ref{lem:newton-deficit-rho-gap} implies that, at each point $x\in\mathbb S^n$,
the values $\rho_i(x)$ are either all equal, or satisfy
$$ \rho_2(x)=\cdots=\rho_{n-1}(x) = \frac12\left(\rho_1(x)+\rho_n(x)\right). $$
Hence the eigenvalues $\mu_i(x)$ of $P|_{x^\perp}$ are either all equal, or satisfy
\begin{equation}\label{midpoint}
\mu_2(x)=\cdots=\mu_{n-1}(x) = \frac12\left(\mu_1(x)+\mu_n(x)\right).
\end{equation}
Let $\lambda_1\le \cdots\le \lambda_{n+1}$ be the eigenvalues of the matrix $P$, with the corresponding orthonormal eigenbasis
$e_1, \ldots, e_{n+1}$. Taking $x=e_i$, the eigenvalues of $P|_{x^\perp}$ are precisely the $n$ numbers obtained from
$\lambda_1, \ldots, \lambda_{n+1}$ by deleting $\lambda_i$.

Taking $x=e_1$ and $x=e_{n+1}$, respectively, we see that the two lists
$$
\lambda_2, \ldots, \lambda_{n+1}
\qquad\text{and}\qquad
\lambda_1, \ldots, \lambda_n
$$
each either consist of equal numbers or satisfy the midpoint condition in \eqref{midpoint}. From this, it is easy to see that $\lambda_1=\cdots=\lambda_{n+1}$. Since $P$ is trace-free, it follows that $P=0$.
Consequently,
$$
h(x)=c+b\cdot x,
\qquad b=(b_1, \ldots, b_{n+1}).
$$
Moreover,
$$
\nabla^2h+h\sigma=c\sigma.
$$
Since the hypersurface is strictly convex, we have $c>0$. Therefore $h$ is the support function of the round sphere of radius $c$ centred at $b$.
\end{proof}

\begin{corollary}
Let $f: M^2 \rightarrow \mathbb{R}^3$ be a smooth, closed, strictly convex embedding. Then
$$
\left(\int_M H d \mu\right)^2-16 \pi|M| \leq \frac{8 \pi}{3} \cdot(4 \pi)^{\frac{1}{3}}\left(\frac{3}{\sqrt{2}}\left|V\left[b_1\right]-V\left[b_2\right]\right|\right)^{\frac{2}{3}} .
$$
\end{corollary}
\begin{proof}
When $n=2$, \eqref{ineq 3} is an equation. The desired inequality then follows from \eqref{ineq 4}, the fact that $(\rho_2-\rho_1)^2=4 \frac{\left(E_1^2-E_2\right)}{E_2^2}$, the H\"older inequality, \eqref{ineq 3} and the Gauss-Bonnet formula.
\end{proof}

Let us end this section by giving another application of Lemma \ref{lem wirtinger}, the spectral Poincaré inequality on $\mathbb{S}^n$.

For a simple closed curve $\gamma$ in $\mathbb R^2$, the Gauss-Bonnet formula and Cauchy-Schwarz inequality gives
$\int_\gamma \kappa^2 d s\ge\frac{4 \pi^2}{L}$. The following result expresses a lower bound of the deficit of this inequality in terms of the isoperimetric deficit, and can also be regarded as a reverse isoperimetric inequality.
\begin{theorem}
Let $\gamma$ be a simple closed $C^2$ curve in $\mathbb{R}^2$ with length $L$, and let $z(s)$ be the positively oriented arclength parametrisation of $\gamma$. Let $\Omega\subset\mathbb{R}^2$
be the domain enclosed by $\gamma$, with area $A$. Let $\nu$ denote the outward-pointing unit normal on $\gamma$, and assume that $\gamma$ has centroid $0$, i.e. $\int_{\gamma}zds=0$. Then
\begin{equation}
\label{ineq reverse isop non convex}
0\le \frac{1}{\pi}(L^2-4\pi A)-\frac{2\pi}{L}\int_\gamma\left|z-\left(\frac{L}{2\pi}\right)\nu\right|^2\, ds
\le \frac14\left(\frac{L}{2\pi}\right)^3 \left(\int_\gamma \kappa^2\, ds-\frac{4\pi^2}{L} \right).
\end{equation}
Equality in either inequalities holds if and only if $\gamma$ is a circle centered at the origin.
\end{theorem}

\begin{proof}
We identify $\mathbb{R}^2$ with $\mathbb{C}$ and use the change of variable
$t=\frac{2\pi}{L}s$. So we can regard $z$ as a map $z=z(t):[0, 2\pi]\to\mathbb{C}$.
Since $\int_0^{2\pi}|\dot z|^2\, dt=\frac{L^2}{2\pi}$ and
$$
2A=\iint_\Omega \operatorname{div}(z)\, dx\, dy
=\int_\gamma x\, dy-y\, dx
=-\int_0^{2\pi}\operatorname{Im}(z\dot{\bar z})\, dt,
$$
it follows that
$$
\frac{1}{\pi}(L^2-4\pi A)
=
\int_0^{2\pi}2\left(|\dot z|^2+\operatorname{Im}(z\dot{\bar z})\right)\, dt.
$$
We rewrite the integrand as
$$
2\left(|\dot z|^2+\operatorname{Im}(z\dot{\bar z})\right)
=
|z+i\dot z|^2+\left(|\dot z|^2-|z|^2\right).
$$
Hence
\begin{equation}
\frac{1}{\pi}(L^2-4\pi A)
=
\int_0^{2\pi}|z+i\dot z|^2\, dt
+
\int_0^{2\pi}\left(|\dot z|^2-|z|^2\right)\, dt.
\label{eq isoperimetric-decomp}
\end{equation}
Now, $z+i\dot z=z-\left(\frac{L}{2\pi}\right)\nu$.
Thus
\begin{equation}
\int_0^{2\pi}|z+i\dot z|^2\, dt
=
\frac{2\pi}{L}\int_\gamma\left|z-\left(\frac{L}{2\pi}\right)\nu\right|^2\, ds.
\label{eq normal-defect}
\end{equation}
Substituting \eqref{eq normal-defect} into \eqref{eq isoperimetric-decomp}, we obtain
\begin{equation}
\int_0^{2\pi}\left(|\dot z|^2-|z|^2\right)\, dt
= \frac{1}{\pi}(L^2-4\pi A) - \frac{2\pi}{L}\int_\gamma\left|z-\left(\frac{L}{2\pi}\right)\nu\right|^2\, ds.
\label{eq first-defect}
\end{equation}
Since the centroid is $0$, we have $\int_0^{2\pi}z(t)\, dt=0$.
Applying the Poincaré-type inequality in Lemma \ref{lem wirtinger}, we get
\begin{equation}
\label{ineq poincare}
0\le
\int_0^{2\pi}\left(|\dot z|^2-|z|^2\right)\, dt
\le
\frac14\int_0^{2\pi}\left(|\ddot z|^2-|\dot z|^2\right)\, dt.
\end{equation}

In view of \eqref{eq first-defect}, we have
\begin{equation}\label{ineq isop second der}
0\le\frac{1}{\pi}\left(L^2-4 \pi A\right)-\frac{2 \pi}{L} \int_\gamma\left|z-\left(\frac{L}{2 \pi}\right) \nu\right|^2 d s\le \frac{1}{4} \int_0^{2 \pi}\left(|\ddot{z}|^2-|\dot{z}|^2\right) d t.
\end{equation}
It remains to express the last integral in geometric terms. Since
$\ddot z(t)=\left(\frac{L}{2\pi}\right)^2\mathbf{\kappa}$.
we have
\begin{equation*}
\int_0^{2\pi}\left(|\ddot z|^2-|\dot z|^2\right)\, dt
=
\left(\frac{L}{2\pi}\right)^3
\left(
\int_\gamma\kappa^2\, ds-\frac{4\pi^2}{L}
\right).
\end{equation*}

Combining this with \eqref{ineq isop second der}, we obtain
$$
0\le
\frac{1}{\pi}(L^2-4\pi A) - \frac{2\pi}{L}\int_\gamma\left|z-\left(\frac{L}{2\pi}\right)\nu\right|^2\, ds \le \frac14\left(\frac{L}{2\pi}\right)^3 \left(\int_\gamma\kappa^2\, ds-\frac{4\pi^2}{L} \right),
$$
which is exactly \eqref{ineq reverse isop non convex}.

If the first inequality in \eqref{ineq reverse isop non convex} is an equality, then from the equality case of \eqref{ineq poincare}, $z(t)=ae^{it}+be^{-it}$
for some $a, b\in\mathbb{C}$. We compute
$$|\dot{z}(t)|^2=|a|^2+|b|^2-a \bar{b} e^{2 i t}-\bar{a} b e^{-2 i t}. $$
On the other hand, $|\dot z(t)|^2=\frac{L^2}{4\pi^2}$ and so by the uniqueness of the Fourier series representation, we have either $a=0$ or $b=0$. Therefore $\gamma$ is a circle centered at the origin.

If the second inequality in \eqref{ineq reverse isop non convex} is an equality, then from the equality case of \eqref{ineq poincare}, $z(t)=\sum_{0<|k|\le 2}a_k e^{ikt}$. By a tedious but straightforward computation,
\begin{equation*}
\begin{aligned}
|\dot{z}(t)|^2
=& \sum_{k, \ell \in\{ \pm 1, \pm 2\}} k \ell a_k \overline{a_{\ell}} e^{i(k-\ell) t}\\
=& \left|a_{-1}\right|^2+\left|a_1\right|^2+4\left|a_{-2}\right|^2+4\left|a_2\right|^2 \\
& +\left(2 a_{-1} \overline{a_{-2}}+2 a_2 \overline{a_1}\right) e^{i t}+\left(2 a_{-2} \overline{a_{-1}}+2 a_1 \overline{a_2}\right) e^{-i t} \\
& +\left(-a_1 \overline{a_{-1}}-4 a_2 \overline{a_{-2}}\right) e^{2 i t}+\left(-a_{-1} \overline{a_1}-4 a_{-2} \overline{a_2}\right) e^{-2 i t} \\
& +\left(-2 a_1 \overline{a_{-2}}-2 a_2 \overline{a_{-1}}\right) e^{3 i t}+\left(-2 a_{-2} \overline{a_1}-2 a_{-1} \overline{a_2}\right) e^{-3 i t} \\
& -4 a_2 \overline{a_{-2}} e^{4 i t}-4 a_{-2} \overline{a_2} e^{-4 i t}.
\end{aligned}
\end{equation*}
By comparing with the Fourier series of $|\dot{z}(t)|^2=\frac{L^2}{4 \pi^2}$, we see that at most one of the coefficients $a_{-2}, a_{-1}, a_1, a_2$ is nonzero. By applying a rigid motion, we deduce that $\dot z(t)=R e^{i t}$ or $\dot z(t)=R e^{i 2t}$. The second case is ruled out as $\gamma$ is simple. It follows that $\gamma$ is a circle centered at the origin.
\end{proof}
\section{Heintze-Karcher inequality in space forms and its deficit}\label{sec Heintze Karcher}
\subsection{An unweighted Heintze-Karcher type inequality}
    We now turn to closed hypersurfaces in space forms $S_K^{n+1}$ satisfying the condition that {{$E_1>\sqrt{-K}$}} if $K \leq 0$. For example, in the Euclidean case, these are mean-convex hypersurfaces. In this case, the first focal length $\rho_1$ is positive, finite, and continuous. It follows that the oriented volume associated with the first focal map is well defined. Our goal is to establish a Heintze--Karcher-type inequality \cite{HeintzeKarcher1978,Ros1987} in all space forms that involves only the mean curvature of the hypersurface and the volume it encloses. More precisely, we shall first show that the oriented volume of the first focal map is non-positive, and then prove that the deficit in the Heintze--Karcher inequality admits a lower bound given by the absolute value of the oriented volume of the first focal map. The inequality obtained here is unweighted, and is therefore different from the weighted Heintze--Karcher inequalities proved in \cite{Brendle2013, QiuXia2015}.

For $K=0,\pm1$, let $S_K^{n+1}$ be the simply connected space form of constant sectional curvature $K$. Denote by $\mathbb{B}_K^{n+1}(r)$ a geodesic ball of radius $r$ in $S_K^{n+1}$, and by $\mathbb{S}_K^n(r)=\partial \mathbb{B}_K^{n+1}(r)$ its boundary geodesic sphere. We define the volume-to-area ratio function by
$$
h_K(r):=\frac{|\mathbb{B}_K^{n+1}(r)|}{|\mathbb{S}_K^n(r)|}.
$$
Explicitly,
\begin{align}\label{eq h_K}
h_K(r)=\frac{\int_0^r s_K(t)^n \,dt}{s_K(r)^n},
\end{align}
where $s_K$ is defined in \eqref{eq sk ck}. Recall also that $c_K=s_K'$.

Let $f:M\to S_K^{n+1}$ be an oriented immersed hypersurface, let $\nu$ be the chosen unit normal field along $f$, and consider its normal exponential map
$$
\Phi:M\times \mathbb{R}\to S_K^{n+1},\qquad
\Phi(x,t):=\exp_{f(x)}(t\nu(x)).
$$
Let $\lambda_1\ge \cdots\ge \lambda_n$ be the principal curvatures of $M$. To describe the corresponding focal maps, we define
$$
\cot_K(t):=\frac{c_K(t)}{s_K(t)}.
$$
We then introduce the function $\arccot _K$, defined on the range of $\cot_K$, as the inverse function of $\cot_K$.
{{When $K=1$, we choose the continuous branch $\arccot_1=\arccot: \mathbb R\to(0,\pi)$.}}
As long as it is well defined, the $i$-th focal length is then given by $\rho_i=\arccot_K\left(\lambda_i\right)$ and the $i$-th focal map is defined by
$$
b_i(x):=\Phi\bigl(x,-\arccot_K(\lambda_i(x))\bigr),\qquad i=1,\dots,n.
$$
In particular, $b_1$ is the first focal map.

\begin{lemma}\label{lem HK1}
For $K=0, \pm1$ and $0<r<\operatorname{diam}(S_K^{n+1})$, let $f_r(t)=c_K(t)-\cot_K(r)s_K(t)$.
Then
$$
h_K(r)=\int_0^r f_r(t)^n\, dt.
$$
\end{lemma}

\begin{proof}
The geodesic sphere $\mathbb{S}_K^n(r)$ is umbilical with principal curvatures $\cot_K(r)$. Hence, the Jacobian for the inward normal exponential map is given by
$$
\left(c_K(t)-\cot_K(r)s_K(t)\right)^n
=
f_r(t)^n.
$$
Therefore,
$$
|\mathbb{B}_K^{n+1}(r)|
=
\int_{\mathbb{S}_K^n(r)}\int_0^r f_r(t)^n\, dt\, dS=\left|\mathbb{S}_K^n(r)\right| \int_0^r f_r(t)^n d t.
$$

Hence
$$
h_K(r)=\frac{|\mathbb{B}_K^{n+1}(r)|}{|\mathbb{S}_K^n(r)|}
=
\int_0^r f_r(t)^n\, dt.
$$
\end{proof}

\begin{lemma}\label{lem V(b1)}
Let $M=\partial \Omega \subset S_K^{n+1}$ be a smooth, closed hypersurface, and let $\lambda_1 \ge \lambda_2 \ge \cdots \ge \lambda_n$ be its principal curvatures. Assume {{$\lambda_1>\sqrt{-K}$}} if $K \le 0$. Then $V\left[b_1\right] \le 0$.
\end{lemma}

\begin{proof}
For each $p\in M$, let $\operatorname{cut}(p)$ denote the cut function of $M$ at $p$ in the inward normal direction.
The first focal length is the smallest positive root $\arccot_K\lambda_1(p)$ of
$$
\prod_{i=1}^n (c_K(t)-\lambda_i(p)s_K(t))=0,
$$
where the LHS is precisely the Jacobian determinant of the normal exponential map in the inward direction.
Hence
$$
\operatorname{cut}(p)\le \arccot_K\lambda_1(p).
$$
Therefore, by the Jacobian formula for the inward normal exponential map,
$$
|\Omega|
=\int_M\int_0^{\operatorname{cut}(p)}\prod_{i=1}^n\left(c_K(t)-\lambda_i s_K(t)\right)\, dt\, d\mu
\le
\int_M\int_0^{\arccot_K\lambda_1}\prod_{i=1}^n\left(c_K(t)-\lambda_i s_K(t)\right)\, dt\, d\mu.
$$
On the other hand, by the definition of $V[b_1]$,
$$
V[b_1]-|\Omega|
=
-\int_M\int_0^{\arccot_K\lambda_1}\prod_{i=1}^n\left(c_K(t)-\lambda_i s_K(t)\right)\, dt\, d\mu.
$$
Hence
$V[b_1]\le 0$.

\end{proof}

When the equality in Lemma \ref{lem V(b1)} holds, then the cut distance is always equal to the first focal distance. When $n\ge2$, this does not imply $M$ is a hypersphere. For example, $M$ can be the ellipsoid $\{ x^2+y^2 +\frac{z^2}{4}=1\}$ in $\mathbb R^3$, where the first conjugate locus and the cut locus coincide. However, when $n=1$, we are able to characterise the equality case:
\begin{proposition}\label{prop v(b1)=0}
Assume the hypotheses of Lemma \ref{lem V(b1)} with $n=1$. Then $V[b_1]=0$ if and only if $M$ is a geodesic circle.
\end{proposition}

\begin{proof}
Parametrise $M$ by $\gamma(s)$, where $s$ is the arclength parameter. The equality case holds if and only if $\mathrm{cut}(s)=\rho(s)$ for all $s$, where $\rho(s)=\arccot_K(k_g(s))$.
Fix $x_0=\gamma(s_0)$ and let $\rho_0:=\rho(s_0)$. Let $e_0:=\exp_{s_0}(-\rho_0\nu(s_0))$ be the focal point of $\gamma$ at $s_0$. Since $\mathrm{cut}(s_0)=\rho_0$, the geodesic segment from $s_0$ to $e_0$ is minimising up to its endpoint, so the geodesic disk of radius $\rho_0$ centred at $e_0$ lies in $\Omega$. Hence $r(s):=d(\gamma(s), e_0)$ has a local minimum at $s_0$.

Let $(r, \theta)$ be the geodesic polar coordinates centred at $e_0$, so that the metric is $dr^2+s_K(r)^2\, d\theta^2$.
Let $\gamma$ be expressed as $(r(s), \theta(s))$ in these coordinates. A direct calculation (\cite[p.~206]{oneill1983semi}) gives
\begin{equation}\label{eq kg}
k_g(s)=r^{\prime}(s) \frac{d}{d s}\left(s_K(r) \theta^{\prime}(s)\right)-s_K(r) \theta^{\prime}(s) r^{\prime \prime}(s)+s_K(r) \theta^{\prime}(s) \operatorname{cot}_K(r)
\end{equation}
In particular, at $s_0$, $r'(s_0)=0$, $s_K(r(s_0))\theta'(s_0)=1$ and $k_g(s_0)=\operatorname{cot}_K(r(s_0))-r''(s_0)$. Since $r(s_0)=\rho_0$ and $k_g(s_0)=\operatorname{cot}_K(\rho_0)$, we obtain $r''(s_0)=0$.

Differentiating $r'(s)^2+(s_K(r(s))\theta'(s))^2\equiv 1$ gives $\left. \frac{d}{ds}\right|_{s=s_0}\left(s_K(r)\theta'\right)=0$, as $\theta'(s_0)\ne0$. Therefore, differentiating \eqref{eq kg} gives $k_g'(s_0)= -r'''(s_0)$. Since $r$ has a local minimum at $s_0$ and both $r'(s_0)$ and $r''(s_0)$ vanish, necessarily $r'''(s_0)=0$. Therefore $k_g'(s_0)=0$. As $x_0$ is arbitrary, $k_g$ is constant along $\gamma$. Thus $\gamma$ is a geodesic circle.
\end{proof}

\begin{theorem}\label{thm hk deficit}
Let $M=\partial\Omega$ be a smooth, closed hypersurface in the simply connected space form $S_K^{n+1}$, where $K=0, \pm1$. Assume that the normalised mean curvature {{$E_1>\sqrt{-K}$}} if $K\le0$. Let $\lambda_1\ge \lambda_2\ge \cdots \ge \lambda_n$ denote the principal curvatures of $M$. Then
$$
\int_M h_K\left(\arccot_K(E_1)\right)\, d\mu-|\Omega|\ge-V[b_1]\ge0.
$$
When $n=1$, the left-hand inequality is an equality. When $n\ge2$, equality in the same inequality holds if and only if $M$ is a hypersphere.
\end{theorem}

\begin{proof}
By the formula \eqref{EnclosedVolumeEqn1},
\begin{align*}
V[b_1]-|\Omega|
=& \int_M\int_0^{-\arccot_K\lambda_1} \prod_{i=1}^n\left(c_K(t)+\lambda_i s_K(t)\right)\, dt\, d\mu\\
=& -\int_M\int_0^{\arccot_K\lambda_1} \prod_{i=1}^n\left(c_K(t)-\lambda_i s_K(t)\right)\, dt\, d\mu.
\end{align*}

For $0\le t\le \arccot_K\lambda_1$, we have
$$
c_K(t)-\lambda_i s_K(t)\ge 0
\qquad \text{for all }i=1, \dots, n.
$$
Hence the AM-GM inequality gives
$$
\prod_{i=1}^n\left(c_K(t)-\lambda_i s_K(t)\right)
\le
\left(c_K(t)-E_1 s_K(t)\right)^n.
$$
Therefore, by Lemma \ref{lem HK1},
\begin{align*}
V[b_1]-|\Omega|
& \ge-\int_M\int_0^{\arccot_K\lambda_1}
\left(c_K(t)-E_1 s_K(t)\right)^n\, dt\, d\mu \\
& \ge-\int_M\int_0^{\arccot_K E_1}
\left(c_K(t)-E_1 s_K(t)\right)^n\, dt\, d\mu \\
& =-\int_M h_K\left(\arccot_K(E_1)\right)\, d\mu.
\end{align*}
Here we used $\lambda_1\ge E_1$, and the fact that $\arccot_K$ is decreasing. Rearranging and applying Lemma \ref{lem V(b1)} gives the result.

If the equality in the first inequality holds, then $M$ is umbilical at every point. Hence, when $n\ge 2$, the equality holds if and only if $M$ is a hypersphere; when $n=1$, the inequality just becomes an identity.
\end{proof}
{{\begin{remark}
When $n=1$, there is only one focal length, and the first focal map $b_1$ is exactly the evolute $e$, whose oriented area $V[e]$ will be discussed in Subsection \ref{sec volume of evolutes in 2d}. In this case $E_1=k_g$, so Theorem \ref{thm hk deficit} becomes an equality:
$$
\int_M h_K\bigl(\arccot_K(k_g)\bigr)\,ds-|\Omega|=-V[e].
$$

Using the formula \eqref{eq h_K} for $h_K$, we obtain
$$
h_0(\arccot_0(k_g))=\frac{1}{2k_g}, \;
h_1(\arccot_1(k_g))=\sqrt{1+k_g^2}-k_g,\;
\text{ and }
h_{-1}(\arccot_{-1}(k_g))=k_g-\sqrt{k_g^2-1}.
$$
Therefore, in the Euclidean, spherical, and hyperbolic cases respectively, we have
$$
\frac12\int_M \frac{1}{k_g}\,ds-|\Omega|=-V[e]\ge0,
$$
$$
\int_M\bigl(\sqrt{1+k_g^2}-k_g\bigr)\,ds-|\Omega|=-V[e]\ge0,
$$
and
$$
\int_M\bigl(k_g-\sqrt{k_g^2-1}\bigr)\,ds-|\Omega|=-V[e]\ge0.
$$

After applying the Gauss--Bonnet formula, this coincides with the result in Proposition \ref{prop space form evolute area}, where the oriented area is computed instead from Proposition \ref{prop oriented vol formula}. Thus Proposition \ref{prop space form evolute area} may be viewed as the one-dimensional Heintze--Karcher deficit formula in this sense.
\end{remark}}}
In the Euclidean case, the inequality in Theorem \ref{thm hk deficit} reads
\begin{align*}
\frac{1}{n+1} \int_{M} \frac{1}{E_1} d \mu-|\Omega| \geq-V\left[b_1\right] \ge 0.
\end{align*}
{{A question concerning the geometric meaning of the three-dimensional Euclidean Heintze--Karcher deficit was raised in \cite{EscuderoReventos2007}. We answer this question by showing that the deficit can be written as the sum of the absolute value of the oriented volume of the first focal map and a weighted integral measuring the umbilicity:
}}
\begin{theorem}\label{thm:HK-exact-remainder-surface}
Let $f:M^2\to\mathbb{R}^3$ be a smooth, closed, mean-convex embedding enclosing a bounded domain $\Omega$, and let $\lambda_1>0$ denote its largest principal curvature. For a continuous function $u$ on $M$, let $b_u(x):=f(x)-u(x)\nu(x)$. Then
{{\begin{align*}
\frac{1}{3} \int_M \frac{1}{E_1} d \mu-|\Omega|
=&\left|V\left[b_{\frac{1}{\lambda_1}}\right]\right|+\frac{1}{6} \int_M \frac{|\stackrel{\circ}{A}|^2}{\lambda_1^2 E_1} d \mu\\
=&-V\left[b_{\frac{1}{E_1}}\right]+\frac{1}{6} \int_M \frac{|\stackrel \circ A|^2}{E_1^3} d \mu.
\end{align*}}}

\end{theorem}

\begin{proof}
When $n=2$, from the proof of Theorem \ref{thm hk deficit},
$$
\left(\frac13\int_M\frac1{E_1}\, d\mu-|\Omega|\right)-\left |V\left[b_{\frac{1}{\lambda_1}}\right]\right|
= \int_M \varepsilon\, d\mu,
$$
where
$\varepsilon := \frac{1}{3E_1} - \int_0^{\frac{1}{\lambda_1}}(1-\lambda_1 t)(1-\lambda_2 t)\, dt $.

A direct calculation gives
$$
\varepsilon=\frac{\left(\lambda_1-\lambda_2\right)^2}{6 \lambda_1^2\left(\lambda_1+\lambda_2\right)}
=\frac{|\stackrel {\circ}{A}|^2}{6 \lambda_1^2 E_1}.
$$
The first identity follows.

On the other hand, we compute
$$
V\left[b_{\frac{1}{E_1}}\right]-|\Omega|=-\frac{1}{3} \int_M \frac{E_2}{E_1^3} d \mu.
$$
From this we get
$$
\frac{1}{3} \int_M \frac{1}{E_1} d \mu-|\Omega|
=-V\left[b_{\frac{1}{E_1}}\right]+\frac{1}{3} \int_M \frac{E_1^2-E_2}{E_1^3} d \mu
=-V\left[b_{\frac{1}{E_1}}\right]+\frac{1}{6} \int_M \frac{|\stackrel{\circ}{A}|^2}{E_1^3} d \mu.
$$
From this, the second identity follows.

\end{proof}
\subsection{An Alexandrov-type rigidity result}

We now apply the unweighted Heintze--Karcher inequality in Theorem~\ref{thm hk deficit} to prove an Alexandrov-type rigidity theorem. Our result identifies a volume--area condition under which the unweighted Heintze--Karcher inequality yields Alexandrov-type rigidity for hypersurfaces with $E_1\equiv c$.

This should be compared with the classical integral approach to Alexandrov-type theorems in space forms. In Euclidean space, the Minkowski formula is unweighted and directly relates the constant mean curvature to the ratio $\frac{|\Omega|}{|M|}$. In non-Euclidean space forms, however, the Minkowski formulas and Heintze--Karcher inequalities involve a conformal factor as a weight. Weighted Heintze--Karcher inequalities and Minkowski formulas have been used to prove Alexandrov-type rigidity results for hypersurfaces of constant mean curvature; see, for instance, Ros~\cite{Ros1987}, Montiel--Ros~\cite{MontielRos1991}, Brendle \cite{Brendle2013}, and Qiu--Xia~\cite{QiuXia2015}.

\begin{theorem}\label{thm Alexandrov}
Let $M=\partial\Omega$ be a smooth, closed, connected hypersurface in the space form $S_K^{n+1}$, where $K=0, \pm1$.
Let $h_K$ be as defined in \eqref{eq h_K}, and set $v_K(x):=h_K\big(\mathrm{arccot}_K(x)\big)$. Then $\Omega$ is a geodesic ball if and only if the normalised mean curvature $E_1\equiv c$, where $v_K(c)= \frac{|\Omega|}{|M|}$.
\end{theorem}

\begin{proof}
Assume first that $E_1\equiv c$, where $v_K(c)=\frac{|\Omega|}{|M|}$. Fix an origin $O$ in $S^{n+1}_K$. Let $p\in M$ be a point at which the distance from $O$ attains its maximum, and set $r_0=d(O, p)$. Then the geodesic sphere $S_{r_0}$ of radius $r_0$ centred at $O$ is tangent to $M$ at $p$ and contains $\Omega$. By the comparison principle, all principal curvatures of $M$ at $p$ are at least $\cot_K(r_0)$. Since $E_1\equiv c$, we obtain $c=E_1(p)\ge \cot_K(r_0)$. In particular, $c>\sqrt{-K}$ when $K\le0$.

Using $E_1\equiv c$ and $v_K(c)=\frac{|\Omega|}{|M|}$, we have
$$
\int_M v_K(E_1)\, d\mu = \int_M v_K(c)\, d\mu = |\Omega|.
$$
On the other hand, Theorem~\ref{thm hk deficit} gives
$$
\int_M v_K\left(E_1\right) d \mu=\int_M h_K\big(\mathrm{arccot}_K(E_1)\big)\, d\mu\ge|\Omega|. $$
Hence when $n\ge2$, by the equality case in Theorem~\ref{thm hk deficit}, $M$ is a hypersphere. When $n=1$, Theorem \ref{thm hk deficit} implies that $V[b_1]=0$ and so $M$ is also a geodesic circle by Proposition \ref{prop v(b1)=0}.
Therefore $\Omega$ is a geodesic ball.

Conversely, suppose that $\Omega$ is a geodesic ball of radius $r$. Then $M=\partial\Omega$ is a geodesic sphere, and hence its normalised mean curvature is constant, with $E_1\equiv \cot_K r=:c$. Moreover,
$$
\frac{|\Omega|}{|M|}
= \frac{|\mathbb B_K^{n+1}(r)|}{|\mathbb S_K^n(r)|}
= h_K(r)=h_K\left(\operatorname{arccot}_K(c)\right)=v_K(c).
$$
\end{proof}

\begin{remark}
\begin{enumerate}
\item The condition $v_K(c)=\frac{|\Omega|}{|M|}$ is necessary in Theorem \ref{thm Alexandrov} because when $K=1$, any closed embedded minimal hypersurface $M$ in $\mathbb S^{n+1}$ which is not an equator satisfies $\frac{|\Omega|}{|M|}<v_1(0)=\frac{\left|\mathbb{B}_1^{n+1}(\frac{\pi}{2})\right|}{\left|\mathbb{S}_1^n(\frac{\pi}{2})\right|}$, thanks to Theorem \ref{thm hk deficit}.
\item
When $K=0$ and $E_1\equiv c$, the Minkowski formula $|M|=\int_M E_1\langle f, \nu\rangle\, d\mu$, together with $\int_M\langle f, \nu\rangle\, d\mu=(n+1)|\Omega|$, gives $c=\frac{|M|} {(n+1)|\Omega| }$. Equivalently, one has $v_0(c)=\frac{|\Omega|}{|M|}$. Thus, in the non-Euclidean space forms, and assuming $M\subset \mathbb S^{n+1}_+$ (hemisphere) when $K=1$, if one could prove under the constant mean curvature assumption $E_1\equiv c$ that $v_K(c)=\frac{|\Omega|}{|M|}$, then Theorem~\ref{thm Alexandrov} would give an alternative proof of the Alexandrov theorem.
\end{enumerate}
\end{remark}
\section{Reverse isoperimetric inequalities in space forms involving the Fenchel deficit}\label{SectionCurvesInSphere}
\subsection{Spherical case}
Let $\mathbb S_+^{n+1}$ be the open northern hemisphere and $r$ be the spherical distance to the north pole.

A closed $C^2$-hypersurface $M $ of $\mathbb{S}_{+}^{n+1}$ is called horoconvex \cite{PanScheuer2025}, if it satisfies
$$
(\cos r) A \ge (1-u) g.
$$
Here $u=\left\langle\sin r \frac{\partial}{\partial r}, \nu\right\rangle$ is the support function, $\nu$ is the unit outward normal, $A$ is the second fundamental form of $M $ and $g$ is its induced metric.
The geometric meaning of this definition is as follows.
If $S$ is a geodesic sphere which lies in the northern hemisphere and touches the equator $\partial \mathbb{S}_{+}^{n+1}$, then on $S$ we have
$$ (\cos r) A =(1-u) g. $$
{{We call these geodesic spheres horospheres, as they are the analogues of horospheres in hyperbolic space. The horoconvexity of $M $ is then equivalent to the requirement that, at every point of $M $, there exists a horosphere contained in the closed northern hemisphere and tangent to the equator that touches $M $ at that point and encloses it.}}

To state our next result, we first define several notions.
Let
$$
\mathrm{V}(r)=\omega \int_0^r \sin^n t d t,
$$
where $\omega=\left|\mathbb{S}^n\right|$. We define $\mathrm{A}(r):=\mathrm{V}^{\prime}(r)=\omega \sin^n(r)$.
Geometrically, $\mathrm{V}(r)$ is the volume of the geodesic ball of radius $r$, while $\mathrm{A}(r)$ is the area of the geodesic sphere of radius $r$, both in $\mathbb{S}_{+}^{n+1}$. For a closed hypersurface $M $ of $\mathbb{S}_{+}^{n+1}$, let $\Omega \subset \mathbb{S}_{+}^{n+1}$ denote the domain enclosed by $M $. We define the isoperimetric deficit of $M $ by
\begin{equation*}
\begin{aligned}
I=I(M):=\mathrm{V} \circ \mathrm{A}^{-1}(|M |)-|\Omega|\ge0.
\end{aligned}
\end{equation*}
Here \(\mathrm A^{-1}\) denotes the inverse of
\(r\mapsto \omega\sin^n r\) on \(0\le r\le \pi/2\).

\begin{theorem}\label{thm1}
Let $M$ be a closed {{horoconvex}} hypersurface in $\mathbb S^{n+1}_+$. Let $\sigma$ be its normalised mean curvature in $\mathbb S^{n+1}_+$ and let $\sigma_E$ be its normalised mean curvature vector in $\mathbb R^{n+2}$. Suppose $M$ encloses a domain $\Omega$ in $\mathbb S^{n+1}_+$. Let
$ I:=\mathrm{V}\circ\mathrm{A}^{-1}(|M|)-|\Omega| $
be the isoperimetric deficit of $M$. Then
\begin{align}\label{ineq horoconvex S}
I\left(2\omega^{\frac{1}{n}}|M|^{\frac{n-1}{n}}
\sqrt{1-\left(\frac{|M|}{\omega}\right)^{\frac{2}{n}}}+I\right)
\le
\left(\int_M |\sigma_E|\right)^2-\omega^{\frac{2}{n}}|M|^{\frac{2n-2}{n}}.
\end{align}

In fact,
\begin{align}\label{remainder ineq}
I\left(2\omega^{\frac{1}{n}}|M|^{\frac{n-1}{n}}
\sqrt{1-\left(\frac{|M|}{\omega}\right)^{\frac{2}{n}}}+I\right)
\le
\left(\int_M |\sigma_E|\right)^2-\omega^{\frac{2}{n}}|M|^{\frac{2n-2}{n}}
-\mathcal R(M),
\end{align}
where
$$
\mathcal R(M):=
\iint_{M\times M}
\frac{(\sigma(x)-\sigma(y))^2}
{\sqrt{(1+\sigma(x)^2)(1+\sigma(y)^2)}+1+\sigma(x)\sigma(y)}
\,d\mu(x)\,d\mu(y)\ge 0.
$$

Equality in \eqref{ineq horoconvex S} holds if and only if $M$ is a geodesic sphere.
\end{theorem}

\begin{remark}
\begin{enumerate}
\item
The inequality \eqref{ineq horoconvex S} can be rewritten as
\begin{align*}
I\left(\frac{2}{n}\mathrm{V}''\circ \mathrm{A}^{-1}(|M |)+I\right)\le \left(\int_{M }\left|\sigma_E\right|\right)^2-\omega^{\frac{2}{n}}|M |^{\frac{2 n-2}{n}} .
\end{align*}
\item
When $n=1$, $I=2 \pi-\sqrt{4 \pi^2-|M |^2}-|\Omega|=2 \pi-\sqrt{4 \pi^2-L^2}-A$. This result is reduced to
$$
L^2-4 \pi A+A^2 \le \left(\int_{M }\left|\sigma_E\right|\right)^2-4 \pi^2.
$$
Indeed, we do not need any convexity assumption in this case. See Theorem \ref{SphereReverseIsoperimetricInequalityTheorem1}.
\end{enumerate}
\end{remark}
The Chen-Fenchel inequality \cite{Chen1971, Fenchel1929} states that for a smooth closed $n$-dimensional submanifold $M $ in a Euclidean space $\mathbb R^{n+m}$, its normalised mean curvature vector $\sigma_E$ satisfies
\begin{align*}
\int_M |\sigma_E|^n\ge \omega.
\end{align*}

As a corollary of Theorem \ref{thm1}, if $M $ is a horoconvex hypersurface of $\mathbb S^{n+1}_+$, we can express the deficit of the Chen-Fenchel inequality by the spherical isoperimetric deficit.
\begin{corollary}\label{cor1}
With the same assumption as Theorem \ref{thm1}, we have
\begin{equation*}\label{fenchel}
I\left(2 \omega^{\frac{1}{n}}|M |^{\frac{n-1}{n}} \sqrt{1-\left(\frac{|M |}{\omega}\right)^{\frac{2}{n}}}+I\right) \le |M |^{\frac{2 n-2}{n}}\left(\left(\int_M  |\sigma_E|^n\right)^{\frac{2}{n}}-\omega^{\frac{2}{n}}\right).
\end{equation*}
\end{corollary}
\begin{proof}[Proof of Corollary \ref{cor1}]
This follows from Theorem \ref{thm1} by applying the H\"older inequality to the term $\left(\int_M\left|\sigma_E\right|\right)^2$.
\end{proof}

\begin{proof}[Proof of Theorem \ref{thm1}]
Observe that on $M $, $\left|\sigma_E\right|= \sqrt{1+\sigma^2}$.
Therefore
\begin{equation}\label{ineq pw}
\begin{split}
& \left(\int_M |\sigma_E|\right)^2-|M |^2-\left(\int_{M } \sigma\right)^2\\
=& \iint_{ M \times M }\left(\sqrt{\left(1+\sigma (x)^2\right)\left(1+\sigma (y)^2\right)}-1-\sigma (x) \sigma (y)\right) d \mu(x) d \mu(y).
\end{split}
\end{equation}

Using the algebraic identity
\begin{equation}\label{eq alg id}
\sqrt{\left(1+a^2\right)\left(1+b^2\right)}-1-a b=\frac{(a-b)^2}{\sqrt{\left(1+a^2\right)\left(1+b^2\right)}+1+a b},
\end{equation}
valid for all $a, b \in \mathbb{R}$, we obtain
$$
\left(\int_{M }\left|\sigma_E\right|\right)^2-|M |^2-\left(\int_{M } \sigma\right)^2=\mathcal{R}(M),
$$
where
$$
\mathcal{R}(M):=\iint_{M  \times M  } \frac{\left(\sigma (x)-\sigma (y)\right)^2}{\sqrt{\left(1+\sigma (x)^2\right)\left(1+\sigma (y)^2\right)}+1+\sigma (x) \sigma (y)} d \mu(x) d\mu(y) \ge 0.
$$

Recall that
$\mathrm{V}(r) =\omega \int_0^r \sin^n t d t$ and $\mathrm{A}(r):=\mathrm{V}'(r)$.
Then $\mathrm{A}(r) =\omega \sin^n(r)$ and $\mathrm{V}^{\prime \prime}(r) =n \omega \sin^{n-1} r \cos r$. The geometric meaning of the latter is that $\mathrm{V}''(r)=n\int_{S(r)}\sigma d\mu$, where $S(r)$ is the geodesic sphere of radius $r$.

By \cite{PanScheuer2025} Theorem 1.4, 
\begin{equation}\label{ineq quermass sphere}
\begin{split}
n \int_{M } \sigma+n|\Omega|
\ge& \mathrm{V}^{\prime \prime}\left(\mathrm{A}^{-1}(|M |)\right)+n\, \mathrm{V}\circ \mathrm{A}^{-1}(|M |)\\
=& n \omega\left(\frac{|M |}{\omega}\right)^{\frac{n-1}{n}}
\sqrt{1-\left(\frac{|M |}{\omega}\right)^{\frac{2}{n}}}
+n\, \mathrm{V}\circ \mathrm{A}^{-1}(|M |).
\end{split}
\end{equation}

As $I:= \mathrm{V}\circ \mathrm{A}^{-1}(|M |)-|\Omega|$,
\begin{align*}
\int_M \sigma
\ge \omega^{\frac{1}{n}} |M |^{\frac{n-1}{n}} \sqrt{1-\left(\frac{|M |}{\omega}\right)^{\frac{2}{n}}}+I.
\end{align*}

Putting this into \eqref{ineq pw}, we have
\begin{align*}
0\le\mathcal{R}(M)\le& \left(\int_{M }\left|\sigma_E\right|\right)^2-|M |^2-\left(\omega^{\frac{1}{n}}|M |^{\frac{n-1}{n}} \sqrt{1-\left(\frac{|M |}{\omega}\right)^{\frac{2}{n}}}+I\right)^2\\
=& \left(\int_{M }\left|\sigma_E\right|\right)^2-\omega^{\frac{2}{n}}|M |^{\frac{2 n-2}{n}}-I\left(2 \omega^{\frac{1}{n}}|M |^{\frac{n-1}{n}} \sqrt{1-\left(\frac{|M |}{\omega}\right)^{\frac{2}{n}}}+I\right).
\end{align*}
The equality in \eqref{ineq horoconvex S} clearly holds if $M $ is a geodesic sphere.
From \eqref{eq alg id}, if the equality case in \eqref{ineq horoconvex S} holds, then $\sigma$ has to be constant. By the Alexandrov theorem on the hemisphere \cite{Alexandrov1962}, we conclude that $M $ is a geodesic sphere.
\end{proof}

In the curve case, we can remove the horoconvexity assumption by replacing the use of the quermassintegral inequality by the Gauss-Bonnet formula.
\begin{theorem}\label{Sphere Reverse Isoperimetric Inequality Theorem1}
Let $\gamma: \mathbb{S}^1 \rightarrow \mathbb{S}^2$ be a smooth, simple, closed curve. Let $L$ be its length, $A$ be the area of the domain enclosed by $\gamma$ whose induced boundary orientation agrees with that of $\gamma$, and $k_g$ its geodesic curvature. Then
$$
L^2-A(4 \pi-A) \leq\left(\int_\gamma \sqrt{1+k_g^2} d s\right)^2-4 \pi^2 .
$$
In fact,
\begin{equation}\label{remainder id}
L^2-A(4 \pi-A)=\left(\int_\gamma \sqrt{1+k_g^2} d s\right)^2-4 \pi^2-\mathcal{R}(\gamma),
\end{equation}
where
$$
\mathcal{R}(\gamma):=\iint_{\gamma \times \gamma} \frac{\left(k_g(s)-k_g(\tau)\right)^2}{\sqrt{\left(1+k_g(s)^2\right)\left(1+k_g(\tau)^2\right)}+1+k_g(s) k_g(\tau)} d s d \tau \geq 0 .
$$
Equality holds if and only if $\gamma$ is a geodesic circle.
\end{theorem}
\begin{proof}
The proof is exactly the same as Theorem \ref{thm1}, except we replace the quermassintegral inequality \eqref{ineq quermass sphere} by the Gauss-Bonnet formula
$\int_\gamma k_g ds +A=2\pi$ to get
$$
\left(\int_\gamma \sqrt{1+k_g^2} d s\right)^2-L^2-(2 \pi-A)^2=\mathcal{R}(\gamma)\ge0.
$$
Rearranging this, we get the result. The equality holds if and only if $k_g$ is constant, which occurs exactly when $\gamma$ is a geodesic circle.
\end{proof}

We have the following oscillation bound for the error in inequality \eqref{ineq horoconvex S}.
\begin{theorem}
Let $M^n \subset \mathbb{S}^{n+1}_+$ be a smooth, closed, embedded horoconvex hypersurface bounding a domain $\Omega$. Let
$I:=\mathrm{V}\circ\mathrm{A}^{-1}(|M|)-|\Omega|$
be its isoperimetric deficit, and let
$\bar{\sigma}:=\frac{1}{|M|}\int_M \sigma\,d\mu$
denote the average of $\sigma$ over $M$. Then
\[
I\left(2\omega^{\frac{1}{n}}|M|^{\frac{n-1}{n}}\sqrt{1-\left(\frac{|M|}{\omega}\right)^{\frac{2}{n}}}+I\right)
\leq
\left(\int_M |\sigma_E|\,d\mu\right)^2
-\omega^{\frac{2}{n}}|M|^{\frac{2n-2}{n}}
-\frac{|M|}{1+\|\sigma\|_\infty^2}\int_M(\sigma-\bar{\sigma})^2\,d\mu.
\]
In the case where $n=1$, we can assume $M$ is a smooth, closed, embedded curve in $\mathbb S^2$ instead, and the inequality becomes
$$
L^2-A(4 \pi-A) \leq\left(\int_M \sqrt{1+k_g^2} d s\right)^2-4 \pi^2-\frac{L}{1+\left\|k_g\right\|_{\infty}^2} \int_M\left(k_g-\bar{k}_g\right)^2 d s,
$$
where $L$ is the length of $M$, $A $ is the area enclosed by $M$, $k_g$ is its geodesic curvature and
$\bar{k}_g:=\frac{1}{L} \int_\gamma k_g d s$.

In both cases, the equality holds if and only if $M$ is a geodesic sphere.
\end{theorem}
\begin{proof}
From the remainder inequality \eqref{remainder ineq} or \eqref{remainder id}, we have
\begin{equation}\label{ineq remainder}
\left(\int_M\left|\sigma_E\right| d \mu\right)^2-\omega^{\frac{2}{n}}|M|^{\frac{2 n-2}{n}}-I\left(2 \omega^{\frac{1}{n}}|M|^{\frac{n-1}{n}} \sqrt{1-\left(\frac{|M|}{\omega}\right)^{\frac{2}{n}}}+I\right)\ge \mathcal{R}(M),
\end{equation}
where
$$
\mathcal{R}(M):=\iint_{M \times M} \frac{(\sigma(x)-\sigma(y))^2}{\sqrt{\left(1+\sigma(x)^2\right)\left(1+\sigma(y)^2\right)}+1+\sigma(x) \sigma(y)} d \mu(x) d \mu(y) \geq 0
$$

Also, for all $x,y\in M$,
$$
\sqrt{\left(1+\sigma(x)^2\right)\left(1+\sigma(y)^2\right)}+1+\sigma(x) \sigma(y) \leq 2\left(1+\|\sigma\|_{\infty}^2\right),
$$
hence
$$ \mathcal{R}(M) \geq \frac{1}{2\left(1+\|\sigma\|_{\infty}^2\right)} \iint_{M \times M}(\sigma(x)-\sigma(y))^2 d \mu(x) d \mu(y). $$
Using the identity
$\frac{1}{2|M|}\iint_{M \times M}(\sigma(x)-\sigma(y))^2 \, d\mu(x)\, d\mu(y)
= \int_M (\sigma-\bar{\sigma})^2\, d\mu$
and substituting this into \eqref{ineq remainder}, we obtain the desired result.

The equality holds if and only if $\sigma$ is constant and so $M$ is a geodesic sphere.
\end{proof}

\subsection{Hyperbolic case}
In the case where $M $ is a hypersurface in hyperbolic space, the hyperbolic quermassintegral inequality \cite{WangXia2014} does not seem to be directly applicable, since it fails to yield the favourable sign needed for our argument. Nevertheless, we are able to establish a related result in dimension one: for a horoconvex curve in the hyperbolic plane, its isoperimetric deficit can be controlled by the Fenchel deficit of the same curve when it is viewed as a spacelike curve in three-dimensional Minkowski spacetime.

We say a closed curve in $\mathbb H^2$ is strictly horoconvex if its geodesic curvature $k_g>1$. We use the hyperboloid model for $\mathbb H^2$:
\begin{align*}
\mathbb H^2=\{(x_0, x_1, x_2)\in \mathbb R^{2, 1}: -x_0^2+x_1^2+x_2^2 =-1 \text{ and }x_0>0\}.
\end{align*}
Notice that, when a strictly horoconvex curve $\gamma \subset \mathbb{H}^2$ is viewed as a strong spacelike curve in $\mathbb{R}^{2,1}$ endowed with the Lorentzian metric $-d x_0^2+d x_1^2+d x_2^2$, its curvature $k_E$ is given by
$$k_E=\sqrt{k_g^2-1}. $$
Here, $\gamma$ is said to be strong spacelike if its unit tangent vector and the curvature vector span a spacelike $2$-plane at each point.
\begin{lemma}\label{lem1}
For $x,y> 1$,
\begin{equation}
xy-1-\sqrt{x^2-1}\sqrt{y^2-1}
=
\frac{(x-y)^2}{xy-1+\sqrt{x^2-1}\sqrt{y^2-1}}.
\end{equation}
In particular, the left-hand side is nonnegative, and it vanishes if and only if $x=y$.
\end{lemma}

\begin{proof}
We compute
\begin{align*}
\left(xy-1-\sqrt{x^2-1}\sqrt{y^2-1}\right)
\left(xy-1+\sqrt{x^2-1}\sqrt{y^2-1}\right)
=&(xy-1)^2-(x^2-1)(y^2-1) \\
=&x^2+y^2-2xy \\
=&(x-y)^2.
\end{align*}
\end{proof}

\begin{theorem}[Reverse isoperimetric inequality on \(\mathbb H^{2}\)]
\label{HyperbolicReverseIsoperimetricInequalityTheorem1}
Let \(\gamma:\mathbb S^{1}\to\mathbb H^{2}\) be a smooth, simple, closed, strictly horoconvex curve. Let \(L\) be its length, \(A\) the area of the enclosed region and \(k_g\) its geodesic curvature. Let \(k_E=\sqrt{k_g^{2}-1}\) be the curvature of \(\gamma\) when regarded as a strong spacelike curve in \(\mathbb R^{2,1}\). Then
\[
L^{2}-A(4\pi+A)\le 4\pi^{2}-\left(\int_{\gamma}k_E\,ds\right)^{2}.
\]
In fact,
\begin{equation*}
L^{2}-A(4\pi+A)
=
4\pi^{2}-\left(\int_{\gamma}k_E\,ds\right)^{2}
-\mathcal R(\gamma),
\end{equation*}
where
\[
\mathcal R(\gamma):=
\iint_{\gamma\times\gamma}
\frac{(k_g(s)-k_g(\tau))^{2}}
{\sqrt{(k_g(s)^{2}-1)(k_g(\tau)^{2}-1)}+k_g(s)k_g(\tau)-1}
\,ds\,d\tau
\ge 0.
\]
Equality holds if and only if \(\gamma\) is a geodesic circle.
\end{theorem}

\begin{proof}
Set
\[
J:=\int_{\gamma}k_E\,ds=\int_{\gamma}\sqrt{k_g^{2}-1}\,ds,
\qquad
K_{\gamma}:=\int_{\gamma}k_g\,ds.
\]
Then
\begin{align*}
J^{2}-K_{\gamma}^{2}+L^{2}
&=
\iint_{\gamma\times\gamma}
\Bigl(
\sqrt{(k_g(s)^{2}-1)(k_g(\tau)^{2}-1)}
-\bigl(k_g(s)k_g(\tau)-1\bigr)
\Bigr)
\,ds\,d\tau.
\end{align*}
Using Lemma \ref{lem1} with \(x=k_g(s)\) and \(y=k_g(\tau)\), we obtain
\[
J^{2}-K_{\gamma}^{2}+L^{2}
=
-\mathcal R(\gamma),
\]
where
\[
\mathcal R(\gamma):=
\iint_{\gamma\times\gamma}
\frac{(k_g(s)-k_g(\tau))^{2}}
{\sqrt{(k_g(s)^{2}-1)(k_g(\tau)^{2}-1)}+k_g(s)k_g(\tau)-1}
\,ds\,d\tau
\ge 0.
\]
Hence
\[
J^{2}-K_{\gamma}^{2}+L^{2}\le 0.
\]

By the Gauss--Bonnet theorem in \(\mathbb H^{2}\),
\[
K_{\gamma}=\int_{\gamma}k_g\,ds=2\pi+A.
\]
Substituting this into the preceding inequality gives
\[
L^{2}-A(4\pi+A)
=
4\pi^{2}-J^{2}-\mathcal R(\gamma)
=
4\pi^{2}-\left(\int_{\gamma}k_E\,ds\right)^{2}-\mathcal R(\gamma),
\]
and therefore
\[
L^{2}-A(4\pi+A)\le 4\pi^{2}-\left(\int_{\gamma}k_E\,ds\right)^{2}.
\]

If equality holds, then \(\mathcal R(\gamma)=0\), so \(k_g\) is constant on
\(\gamma\). Hence \(\gamma\) is a geodesic circle. The converse is immediate.
\end{proof}

\subsection{Relations to the volume of evolutes in the two-dimensional case}\label{sec volume of evolutes in 2d}
Recall that $S^2_K$ is the two-dimensional simply connected space form of curvature $K$. Here we focus on the case where $K=0,\pm1$.
\begin{definition}[Evolute in a two-dimensional space form]
Let $f:M\to S^2_K$ be a smooth, closed, embedded curve with unit normal $\nu$, and geodesic curvature $k_g$ with respect to $\nu$. Suppose $k_g>\sqrt{-K}$ if $K \leq 0$. We then define
$$
u(x):=-\arccot_K(k_g(x)),
$$
where $\arccot_K$ is defined in Section \ref{sec Heintze Karcher}.
More explicitly,
$$
u=
\begin{cases}
-\arccot(k_g), & K=1, \\
-\dfrac{1}{k_g}, & K=0, \\
-\mathrm{arccoth}(k_g), & K=-1.
\end{cases}
$$
The evolute of $M$ is then the map
\begin{equation}\label{eq evolute}
e:M\to S^2_K, \qquad e(x):=\exp_{f(x)}\left(u(x)\nu(x)\right).
\end{equation}
\end{definition}

\begin{proposition}[Oriented area of the evolute in a two-dimensional space form]\label{prop space form evolute area}
Let $M^1\subset S^2_K$ be a smooth, closed, embedded curve bounding a domain $\Omega$ whose unit outward normal is $\nu$, and let $k_g$ denote its geodesic curvature with respect to $\nu$. We assume $k_g>\sqrt{-K}$ if $K\le 0$. For the evolute given by \eqref{eq evolute}, the oriented area enclosed by it is
\begin{align*}
V_2^K[e]
=
\begin{cases}
|\Omega|+\frac{1}{K}\int_M \left({k_g-\sqrt{k_g^2+K}}\right)\, d\mu,
\qquad & \text{if } K=\pm1\\
|\Omega|-\frac{1}{2} \int_M \frac{1}{k_g} d \mu,
\qquad & \text{if }K=0.
\end{cases}
\end{align*}

Moreover,
\begin{enumerate}
\item
For $K=1$, we have
\begin{equation*}
V_2^1[e]
= 2\pi-\int_M\sqrt{1+k_g^2}\, d\mu \le 0.
\end{equation*}
\item
For $K=-1$, we have
\begin{equation*}
V_2^{-1}[e]
= \int_M\sqrt{k_g^2-1}\, d\mu-2\pi \le 0.
\end{equation*}
\item
For $K=0$, we have
\begin{equation*}
V_2^0[e]
=
|\Omega|-\frac{1}{2}\int_M\frac{1}{k_g}\, d\mu\le0.
\end{equation*}
\end{enumerate}

\end{proposition}

\begin{proof}
For $n=1$, the oriented volume formula in a space form reads
$$
V_2^K[u]
=
|\Omega|+\int_M\sum_{i=0}^1 H_{1-i}\phi_i^K(u)\, d\mu
=
|\Omega|+\int_M\left(k_g\phi_0^K(u)+\phi_1^K(u)\right)\, d\mu.
$$
Moreover,
$$
\phi_0^K(s)=\int_0^s s_K(t)\, dt,
\qquad
\phi_1^K(s)=\int_0^s c_K(t)\, dt=s_K(s).
$$
Since $c_K'(t)=-K\, s_K(t)$, we also have
$$
\phi_0^K(s)=
\begin{cases}
\dfrac{1-c_K(s)}{K}, & K\ne 0, \\[1ex]
\dfrac{s^2}{2}, & K=0.
\end{cases}
$$
Hence
$$
V_2^K[u]
=
|\Omega|
+
\int_M\left(k_g\phi_0^K(u)+s_K(u)\right)\, d\mu.
$$

Now set $u=-\arccot_K(k_g)$.
Using $c_K(u)^2+K\, s_K(u)^2=1$,
we obtain
$$
s_K(u)= -\frac{1}{\sqrt{k_g^2+K}},
\qquad
c_K(u)= \frac{k_g}{\sqrt{k_g^2+K}}.
$$

If $K\ne 0$, then
\begin{align*}
k_g\phi_0^K(u)+s_K(u)
& =
k_g\frac{1-c_K(u)}{K}+s_K(u)\\
& =
\frac{k_g}{K}\left(1-\frac{k_g}{\sqrt{k_g^2+K}}\right)-\frac{1}{\sqrt{k_g^2+K}}\\
& = \frac{k_g-\sqrt{k_g^2+K}}{K}.
\end{align*}
Thus, for $K=\pm1$,
$$
V_2^K[e]
=|\Omega|+\frac{1}{K}\int_M \left({k_g-\sqrt{k_g^2+K}}\right)\, d\mu.
$$
For $K=0$, we have
$u=-\frac{1}{k_g}$, $\phi_0^0(u)=\frac{u^2}{2}$, and $\phi_1^0(u)=u$,
so
$$
k_g\phi_0^0(u)+\phi_1^0(u)
=
k_g\frac{u^2}{2}+u
=-\frac{1}{2k_g}.
$$
Hence
$$
V_2^0[e]
=
|\Omega|-\frac{1}{2}\int_M \frac{1}{k_g}\, d\mu.
$$
This is non-positive by the Heintze-Karcher inequality \cite{HeintzeKarcher1978,Ros1987}.

In the case where $K=1$, we have
$$
V_2^1[e]
=
|\Omega|+\int_M \left(k_g-\sqrt{1+k_g^2}\right)\, d\mu.
$$
By the Gauss--Bonnet theorem on $\mathbb S^2$,
$$
|\Omega|+\int_M k_g\, d\mu=2\pi,
$$
and hence
$$
V_2^1[e]
=
2\pi-\int_M\sqrt{1+k_g^2}\, d\mu.
$$
Finally, $\sqrt{1+k_g^2}$ is the curvature of $M$ when regarded as a space curve in $\mathbb R^3$. Therefore, by the Euclidean Fenchel inequality, we have
$$
V_2^1[e]\le 0.
$$

For $K=-1$, we have
\begin{equation*}
V_2^{-1}[e]
=
|\Omega|+\int_M\left(\sqrt{k_g^2-1}-k_g\right)\, d\mu.
\end{equation*}
By Gauss--Bonnet theorem,
\begin{equation*}
\int_M k_g\, d\mu=|\Omega|+2\pi,
\end{equation*}
and hence
\begin{equation*}
V_2^{-1}[e]
=
-2\pi+\int_M\sqrt{k_g^2-1}\, d\mu.
\end{equation*}
Notice that when a strictly horoconvex $M \subset \mathbb{H}^2$ is regarded as a strong spacelike curve in $\mathbb R^{2,1}$, then its curvature $k_E$ is given by
$$
k_E=\sqrt{k_g^2-1}.
$$
Since $M$ is a simple closed curve on $\mathbb H^2$, it is of index $1$ when regarded as a curve in $\mathbb R^{2,1}$ \cite{YeMa2019}. Indeed, it winds around any timelike axis passing through any interior point in $\Omega$ exactly once.
By \cite{YeMa2019} Theorem 1.1, we know that
\begin{align*}
\int_\gamma k_E ds \le 2\pi.
\end{align*}
Therefore
\begin{equation*}
V_2^{-1}[e]=-2 \pi+\int_M \sqrt{k_g^2-1} d \mu\le0.
\end{equation*}

\end{proof}

\begin{corollary}\label{cor isop evolute}
Let $M$ be a smooth, simple, closed curve in $\mathbb{S}^2$. Let $L$ be its length, let $A\in(0, 2\pi)$ be the area enclosed by $M$, let $k_g$ be its geodesic curvature, and let $V[e]\le 0$ denote the oriented area enclosed by its spherical evolute. Then
$$
L^2-A(4\pi-A)\le |V[e]|(4\pi+|V[e]|).
$$
In fact,
$$
L^2-A(4\pi-A)=|V[e]|(4\pi+|V[e]|)-\mathcal{R}(M),
$$
where
$$
\mathcal{R}(M):=\iint_{M\times M}
\frac{\left(k_g(s)-k_g(\tau)\right)^2}
{\sqrt{\left(1+k_g(s)^2\right)\left(1+k_g(\tau)^2\right)}+1+k_g(s)k_g(\tau)}
\, ds\, d\tau \ge 0.
$$
\end{corollary}

\begin{proof}
By Proposition \ref{prop space form evolute area},
$$
V[e]=2\pi-\int_M \sqrt{1+k_g^2}\, ds\le0.
$$
Substituting this
into Theorem \ref{Sphere Reverse Isoperimetric Inequality Theorem1} yields
\begin{align*}
L^2-A(4\pi-A)
=& \left(2\pi+|V[e]|\right)^2-4\pi^2-\mathcal{R}(M)\\
=& |V[e]|(4\pi+|V[e]|)-\mathcal{R}(M).
\end{align*}
\end{proof}.

\begin{corollary}\label{cor isop evolute hyperbolic}
Let $M$ be a smooth, simple, closed strictly horoconvex curve in $\mathbb{H}^2$. Let $L$ be its length, let $A$ be the area enclosed by $M$, let $k_g$ be its geodesic curvature, and let $V[e]\le 0$ denote the oriented area enclosed by its hyperbolic evolute. Then
$$
L^2-A(4\pi+A)\le |V[e]|(4\pi-|V[e]|).
$$
In fact,
$$
L^2-A(4\pi+A)=|V[e]|(4\pi-|V[e]|)-\mathcal{R}(M),
$$
where
$$
\mathcal{R}(M):=\iint_{M\times M}
\frac{\left(k_g(s)-k_g(\tau)\right)^2}
{\sqrt{\left(k_g(s)^2-1\right)\left(k_g(\tau)^2-1\right)}+k_g(s)k_g(\tau)-1}
\, ds\, d\tau \ge 0.
$$
\end{corollary}

\begin{proof}
By Proposition \ref{prop space form evolute area},
$$
V[e]=\int_M \sqrt{k_g^2-1}\, ds-2\pi\le0.
$$
Hence
$$
\int_M \sqrt{k_g^2-1}\, ds=2\pi-|V[e]|.
$$
Substituting this into the hyperbolic reverse isoperimetric inequality yields
\begin{align*}
L^2-A(4\pi+A)
=& 4\pi^2-\left(2\pi-|V[e]|\right)^2-\mathcal{R}(M)\\
=& |V[e]|(4\pi-|V[e]|)-\mathcal{R}(M).
\end{align*}
\end{proof}

\subsection{Volumes of spherical focal maps and an integral measure of umbilicity}
In this subsection we return to the spherical case. Thus, throughout we take $K=1$, and consider a smooth, closed, oriented embedded hypersurface
$$
f:M^n\to \mathbb S^{n+1}.
$$
Accordingly, we suppress the superscript $K$ from the notation introduced earlier. In particular, for $i=0,\dots,n$ we write (cf. \eqref{eq phi_i})
$$
\phi_i(s):=\phi_i^1(s)=\int_0^s (\sin t)^{\,n-i}(\cos t)^i\,dt,
$$
and we write $V$ in place of $V_{n+1}^1$.

Our aim here is to establish a spherical analogue of the Euclidean estimate \eqref{ineq 3} in Theorem~\ref{thm euclidean 2} for focal maps. More precisely, we shall show that the difference between the oriented volumes of the two focal maps corresponding to the smallest and largest radii of curvature is controlled by an integral quantity measuring the umbilicity.

Let $\nu$ be the chosen unit normal field along $f$, and then we can define the $i$-th focal map $b_i$ as in Section \ref{sec Heintze Karcher}, by setting $K=1$.
These maps are continuous, and hence their oriented volumes, as defined in Definition~\ref{def ori vol}, are well defined. Moreover, they are given by \eqref{EnclosedVolumeEqn1}.

\begin{lemma}\label{lem sphere1}
Let $\lambda_1\ge \lambda_2\ge \cdots \ge \lambda_n$, let $\sigma_k$ be the $k$-th elementary symmetric polynomial in $\lambda_1, \dots, \lambda_n$, and define
\begin{equation}\label{eq F}
F(u):=\sum_{i=0}^n \sigma_{n-i}\, \phi_i(u).
\end{equation}
Then
$$
\left|F\left(-\cot^{-1}\lambda_1\right)-F\left(-\cot^{-1}\lambda_n\right)\right|
\le \frac{1}{n(n+1)}\left(\lambda_1-\lambda_n\right)^{n+1}.
$$
\end{lemma}

\begin{proof}
Since $K=1$, we have
$$\phi_i'(t)=(\sin t)^{\, n-i}(\cos t)^i. $$
Hence
$$
F'(t)=\sum_{i=0}^n \sigma_{n-i}(\sin t)^{\, n-i}(\cos t)^i
=\prod_{j=1}^n\left(\cos t+\lambda_j\sin t\right).
$$
Therefore
$$
F\left(-\cot^{-1}\lambda_1\right)-F\left(-\cot^{-1}\lambda_n\right)
=
-\int_{-\cot^{-1}\lambda_1}^{-\cot^{-1}\lambda_n}
\prod_{j=1}^n\left(\cos t+\lambda_j\sin t\right)\, dt.
$$
Using the substitution $u:=-\cot t$, we obtain
$$
F\left(-\cot^{-1}\lambda_1\right)-F\left(-\cot^{-1}\lambda_n\right)
=(-1)^n \int_{\lambda_n}^{\lambda_1}
\frac{\prod_{j=1}^n(\lambda_j-u)}{(1+u^2)^{\frac n2+1}}\, du.
$$
Consequently,
\begin{align*}
\left|F\left(-\cot^{-1}\lambda_1\right)-F\left(-\cot^{-1}\lambda_n\right)\right|
& \le
\int_{\lambda_n}^{\lambda_1}
\frac{\prod_{j=1}^n|\lambda_j-u|}{(1+u^2)^{\frac n2+1}}\, du\\
& \le
\int_{\lambda_n}^{\lambda_1}\prod_{j=1}^n|\lambda_j-u|\, du.
\end{align*}
We may now proceed exactly as in Lemma~\ref{lem:F-rho-gap}, and conclude that
$$
\left|F\left(-\cot^{-1}\lambda_1\right)-F\left(-\cot^{-1}\lambda_n\right)\right|
\le
\frac{1}{n(n+1)}\left(\lambda_1-\lambda_n\right)^{n+1}.
$$
\end{proof}

\begin{lemma}\label{lem sphere2}
Let $\lambda_1\ge \lambda_2\ge \cdots \ge \lambda_n$, and let $E_1:=\frac{\sigma_1}{n}$, $E_2:=\frac{\sigma_2}{\binom n2}$,
where $\sigma_i$ is the $i$-th elementary symmetric polynomial in $\lambda_1, \dots, \lambda_n$. Then
$$
(\lambda_1-\lambda_n)^{n+1}
\le
\left({{2n}}(n-1)\right)^{\frac{n+1}{2}}
\left(E_1^2-E_2\right)^{\frac{n+1}{2}}.
$$
\end{lemma}

\begin{proof}
We use the identity
$$
\sum_{1\le i<j\le n}(\lambda_i-\lambda_j)^2
=
n^2(n-1)\left(E_1^2-E_2\right).
$$
As in \eqref{var to maxmin},
$$
{{\frac{n}{2}}}(\lambda_1-\lambda_n)^2
\le
\sum_{1\le i<j\le n}(\lambda_i-\lambda_j)^2
=
n^2(n-1)\left(E_1^2-E_2\right).
$$
Raising both sides to the power $\frac{n+1}{2}$ gives the result.
\end{proof}

\begin{theorem}\label{thm sphere umbilic}
Let $f:M^n\to \mathbb S^{n+1}$ be a smooth, closed, oriented, embedded hypersurface ($n\ge2$), and let $E_k$ denote the normalised $k$-th mean curvature of $M$. Let $b_i$ be the focal maps associated with the principal curvatures $\lambda_1\ge \cdots \ge \lambda_n$. Then
$$
\frac{n(n+1)}{\left({{2n}}(n-1)\right)^{\frac{n+1}{2}}}
\left|V[b_n]-V[b_1]\right|
\le
\int_M (E_1^2-E_2)^{\frac{n+1}{2}}\, d\mu.
$$
\end{theorem}

\begin{proof}
By \eqref{EnclosedVolumeEqn1},
$$
V[b_i]
=
|\Omega|+\int_M F(-\cot^{-1}\lambda_i)\, d\mu,
$$
where $F(u)=\sum_{i=0}^n H_{n-i} \phi_i(u)$. Hence, by Lemma~\ref{lem sphere1} and Lemma~\ref{lem sphere2},
\begin{align*}
\left|V[b_n]-V[b_1]\right|
& =
\left|\int_M\left(F(-\cot^{-1}\lambda_n)-F(-\cot^{-1}\lambda_1)\right)\, d\mu\right|\\
& \le
\int_M\left|F(-\cot^{-1}\lambda_n)-F(-\cot^{-1}\lambda_1)\right|\, d\mu\\
& \le
\int_M \frac{1}{n(n+1)}(\lambda_1-\lambda_n)^{n+1}\, d\mu\\
& \le
\int_M \frac{\left({{2n}}(n-1)\right)^{\frac{n+1}{2}}}{n(n+1)}
\left(E_1^2-E_2\right)^{\frac{n+1}{2}}\, d\mu.
\end{align*}
This proves the theorem.
\end{proof}
\begin{remark}
It is natural to ask whether the analogue of Theorem~\ref{thm sphere umbilic}, or of the Euclidean estimate \eqref{ineq 3} in Theorem~\ref{thm euclidean 2}, also holds in hyperbolic space. However, our argument does not appear to extend to that setting. More specifically, an estimate analogous to Lemma~\ref{lem sphere1} or Lemma~\ref{lem:F-rho-gap} seems to fail in the hyperbolic case.

The closest bound we have been able to obtain is that if $1<\lambda_j=\coth \rho_j$ for $j=1, \cdots, n$, with $0<\rho_1 \leq \rho_2 \leq \cdots \leq \rho_n$, then
$$
\left|F\left(-\rho_1\right)-F\left(-\rho_n\right)\right| \leq \frac{\left(\rho_n-\rho_1\right) \sinh ^n\left(\rho_n-\rho_1\right)}{n(n+1) \prod_{j=1}^n \sinh \rho_j} .
$$
where $F(u):=\sum_{i=0}^n \sigma_{n-i} \phi_i^{-1}(u)$.
This suggests that, in hyperbolic space, the difference between the oriented volumes of the focal maps may not admit a bound of the same type as \eqref{ineq 3} in Theorem~\ref{thm euclidean 2} or Theorem~\ref{thm sphere umbilic} in terms of the integral quantity measuring deviation from umbilicity.
\end{remark}

\appendix
\section{\ }\label{SectionAppendix}

\subsection{Standard results in the Minkowski plane}\label{SubsectionMinkowskiAppendix}

\begin{proposition}\label{SubsectionMinkowskiAppendixPropostion1}
Suppose that \(\gamma:\mathbb S^{1}\to\mathcal M^{2}\) is a smooth, simple,
closed, strictly convex curve in the Minkowski plane with associated
indicatrix \(\partial\mathcal U\) and isoperimetrix \(\mathcal I\). Then
\begin{equation}
\int_{\gamma}\kappa\,d\sigma=2\mathscr A(\mathcal I).
\label{SubsectionMinkowskiAppendixPropostionEqn1}
\end{equation}
\end{proposition}
\begin{proof}
Using \(d\sigma=h\,ds\), \(d\theta=k\,ds\), and
\(\kappa=k(h_{\theta\theta}+h)\), we obtain
\[
\int_{\gamma}\kappa\,d\sigma
=
\int_{0}^{2\pi}h(h_{\theta\theta}+h)\,d\theta.
\]
On the other hand, the isoperimetrix is parametrised by
\[
\mathcal I(\theta)=-h_{\theta}\tau+h\,n,
\qquad
\tau=(\cos\theta,\sin\theta),
\quad
n=(-\sin\theta,\cos\theta),
\]
so
\[
\mathcal I_{\theta}=-(h_{\theta\theta}+h)\tau.
\]
Therefore
\[
2\mathscr A(\mathcal I)
=
\int_{0}^{2\pi}\mathcal I\wedge \mathcal I_{\theta}\,d\theta
=
\int_{0}^{2\pi}h(h_{\theta\theta}+h)\,d\theta,
\]
which proves \eqref{SubsectionMinkowskiAppendixPropostionEqn1}.
\end{proof}

\subsection{A spectral Poincar\'e inequality on \texorpdfstring{\(\mathbb S^{n}\)}{Sn}}
\begin{lemma}\label{lem wirtinger}
Let $u: \mathbb{S}^n \rightarrow \mathbb{R}$ be smooth and satisfy $\int_{{\mathbb{S}}^n} u d \sigma=0$. Then
$$
0 \le \frac{1}{n} \int_{\mathbb{S}^n}|\overline{\nabla} u|^2 d \sigma-\int_{\mathbb{S}^n} u^2 d \sigma \le \frac{1}{2(n+1)}\left(\frac{1}{n} \int_{\mathbb{S}^n}(\overline{\Delta} u)^2 d \sigma-\int_{\mathbb{S}^n}|\overline{\nabla} u|^2 d \sigma\right).
$$
Equality in the upper bound holds if and only if $u$ has no spherical harmonic components of degree $k \ge 3$.
\end{lemma}

\begin{proof}
The first inequality is just the classical Poincar\'e inequality on the sphere, so let us prove the second inequality.
Let $\widetilde{\Delta}=-\overline\Delta$. It is known that the $k$-th eigenvalue of $\widetilde \Delta$ is given by $\lambda_k=k(k+n-1)$, $k\ge0$.
Let $\langle\cdot, \cdot\rangle_{\sigma}$ and $\|\cdot\|$ denote the $L^2$ inner product and the $L^2$ norm on $\mathbb S^n$.

Since $\int_{\mathbb S^n}u=0$, the function $u$ has no component in the eigenspace corresponding to $\lambda_0=0$, $\left(\widetilde{\Delta}-\lambda_1\right)\left(\widetilde{\Delta}-\lambda_2\right)$ is non-negative on the orthogonal complement of this subspace. Therefore
\begin{align*}
0
& \le \big\langle u, (\widetilde{\Delta}-\lambda_1)(\widetilde{\Delta}-\lambda_2)u\big\rangle_{\sigma}\\
& = \big\langle u, (\widetilde{\Delta}-\lambda_1)\widetilde{\Delta}u\big\rangle_{\sigma}-\lambda_2\big\langle u, (\widetilde{\Delta}-\lambda_1)u\big\rangle_{\sigma}.
\end{align*}

Using integration by parts,
$\langle u, \widetilde{\Delta}u\rangle_{\sigma}=\int_{\mathbb S^n}|\overline\nabla u|^2=\|\overline \nabla u\|^2$ and $\langle u, \widetilde{\Delta}^2u\rangle_{\sigma}=\int_{\mathbb S^n}(\overline\Delta u)^2=\|\overline \Delta u\|^2$. So
\begin{align*}
\lambda_2\left(\|\overline \nabla u\|^2-\lambda_1\|u\|^2\right)\le \|\overline\Delta u\|^2-\lambda_1\|\overline \nabla u\|^2,
\end{align*}
which is the desired inequality.

The equality holds if and only if the spherical harmonics decomposition of $u$ does not contain any mode with $k\ge 3$.
\end{proof}

\begin{remark}[Recovery of Hurwitz's inequality]\label{IsoTheoremRemark1}
Let \(\gamma\subset\mathbb R^{2}\) be a smooth strictly convex curve with support
function \(h\). Then
\[
L(\gamma)=\int_{0}^{2\pi}h\,d\theta,
\qquad
A(\gamma)=\frac12\int_{0}^{2\pi}h(h''+h)\,d\theta.
\]
Applying Lemma~\ref{lem wirtinger} with \(n=1\) to
\(u:=h-\bar h\), where
\(\bar h:=\frac{1}{2\pi}\int_{0}^{2\pi}h\,d\theta\), yields
\[
0
\le
\int_{0}^{2\pi}(h')^{2}\,d\theta
-
\int_{0}^{2\pi}(h-\bar h)^{2}\,d\theta
\le
\frac14\left(
\int_{0}^{2\pi}(h'')^{2}\,d\theta
-
\int_{0}^{2\pi}(h')^{2}\,d\theta
\right).
\]
Integrating by parts and using the formulas above, this becomes
\[
0
\le
\frac{1}{2\pi}\bigl(L(\gamma)^{2}-4\pi A(\gamma)\bigr)
\le
\frac14\int_{0}^{2\pi}h''(h''+h)\,d\theta.
\]
If \(e:=\gamma+\kappa^{-1}\nu\) is the evolute, then the Euclidean
specialisation of Lemma~\ref{MinkGraphAreaLemma1} gives
\[
A(e)
=
A(\gamma)-\frac12\int_{\gamma}\kappa^{-1}\,ds
=
-\frac12\int_{0}^{2\pi}h''(h''+h)\,d\theta.
\]
Hence
\[
0\le L(\gamma)^{2}-4\pi A(\gamma)\le \pi|A(e)|.
\]
Moreover, equality in the upper bound holds if and only if \(h-\bar h\) has no
Fourier modes of order \(\ge3\), equivalently if and only if
\[
h(\theta)=a_{0}+a_{1}\cos\theta+b_{1}\sin\theta+a_{2}\cos2\theta+b_{2}\sin2\theta.
\]
This recovers Hurwitz's inequality \eqref{HurwitzInequality}.
\end{remark}

\bibliographystyle{plain}
\bibliography{ReverseIsobib}

\end{document}